\documentclass[12pt,one side]{article}
\usepackage{amsmath, amsthm, amscd, amsfonts, amssymb, graphicx, color}
\usepackage[bookmarksnumbered, colorlinks, plainpages]{hyperref}
\usepackage{mathrsfs}
\usepackage{algorithm2e}
\usepackage{hhline,booktabs}
\usepackage{multirow}
\usepackage{float}
\usepackage{placeins}
\usepackage{subfigure}
\usepackage{lipsum}
\usepackage{upgreek}
\usepackage{hhline,booktabs}
\usepackage{multirow}
\usepackage{subfigure}
\usepackage[mathscr]{euscript}
\usepackage{multirow,bigstrut,rotating,float}
 \usepackage{algorithmic}
 \usepackage{kantlipsum}

\allowdisplaybreaks
\numberwithin{equation}{section}
\newtheorem{theorem}{Theorem}[section]
\newtheorem{cor}[theorem]{Corollary}

\newtheorem{prop}[theorem]{Proposition}
\newtheorem{defn}[theorem]{Definition}
\newtheorem{rem}[theorem]{\bf{Remark}}
\newtheorem{example}[theorem]{\bf{Example}}

\begin{document}

\title{Biframes and some of their properties}

\maketitle
%\tableofcontents
\begin{center}
	\textbf{M. Firouzi Parizi}, \textbf{A. Alijani}, \textbf{M.A. Dehghan},\\ [0.2cm]
	\small {\textit{Department of Mathematics, Faculty of Mathematics, 
	\\
	Vali-e-Asr University of Rafsanjan, Rafsanjan, Iran}}\\
	
 \texttt{{f75maryam@gmail.com; Alijani@vru.ac.ir; dehghan@vru.ac.ir }}
 \end{center}

%------------------------------------------------------------------------------------%

\begin{abstract} 
Recently, frame multipliers, pair frames, and controlled frames have been investigated to improve the numerical efficiency of iterative algorithms for inverting the frame operator and other applications of frames. 
In this paper, the concept of biframe is introduced for a Hilbert space. A biframe is a pair of sequences in a Hilbert space that applies to an inequality similar to frame inequality. Also, it can be regarded as a generalization of controlled frames and a special kind of pair frames. The basic properties of biframes are investigated based on the biframe operator. Then, biframes are classified based on the type of their constituent sequences. In particular, biframes for which one of the constituent sequences is an orthonormal basis $\{e_k\}_{k=1}^\infty$ are studied. Then, a new class of Riesz bases denoted by $[\{e_k\}]$, is introduced and is called b-Riesz bases. An interesting result is also proved, showing that the set of all b-Riesz bases is a proper subset of the set of all Riesz bases. More precisely, b-Riesz bases induce an equivalence relation on $[\{e_k\}]$.
\end{abstract}

2\noindent \textit{2010 AMS Subject Classification}: 65F10, 15A09.

\bigskip
\section{Introduction}

Many years after the advent of frame theory, its importance and applications in various scientific fields have become clear to everyone. The aim of this theory, developed by Duffin and Schaeffer in 1952 \cite{Duffin}, was to solve some problems related to the non-harmonic Fo8urier series. In fact, the frame theory developed Gabor's studies in signal analysis in a pure form, which was extensively studied by the fundamental paper of Daubechies, Grossmann, and Meyer \cite{D.G.M} in 1986. Many researchers in various fields of pure and applied mathematics, engineering, medicine, etc. have studied frames. For more information on frame theory and its applications, we refer the readers to \cite{chris,Daub,Feich, Han,Young}.

To further apply this theory, many generalizations of frames have been proposed. For example, \emph{generalized frames} (\emph{g-frames}),  introduced by Sun \cite{sun1,sun2}, generalize not only the original concept of frame, but also other generalizations of frames, including \emph{bounded quasi-projectors} \cite{Forna1,Forna2}, \emph{frames of subspaces} \cite{Asgari,Casa}, \emph{pseudo-frames} \cite{Li}, \emph{oblique frames} \cite{Chris3,Eldar}, and \emph{outer frames} \cite{Aldroubi}. After that, Askarizadeh and Dehghan showed that every  \emph{g-frame} is an special frame \cite{Dehghan3}. Also, frames and their generalization in Hilbert $C^*$-modules and locally compact abelian groups have been studied in \cite{Dehghan1, Dehghan2,Azarmi1}. \emph{Controlled frames} \cite{Balazs} provide another generalization of frames, actually developing theories related to this concept which were previously introduced in \cite{Bogda} 2and used only as a tool for spherical wavelets.

\emph{pair frames} are a concept for a pair of sequences in a Hilbert space, introduced in \cite{safapour}, 
and one of their important results is to obtain a new reconstruction formula for members of the Hilbert space.
the concept proposed in this paper aims at a different study of a pair of sequences in a Hilbert space, that is called a \emph{biframe}. To define a frame, only one sequence is used; but, to define a biframe two sequences are needed. In fact, the concept of \emph{biframe} is proposed as a generalization of controlled frames and a special case of pair frames. A \emph{biframe} is a pair $(\{f_k\}_{k=1}^\infty, \{g_k\}_{k=1}^\infty)$ of sequences in the Hilbert space $\mathcal{H}$, if there exist positive constants $A$ and $B$ such that
\begin{equation*}
A{\Vert f \Vert}^2\leq\sum_{k=1}^\infty \langle f,f_{k}\rangle\langle g_{k},f\rangle\leq B{\Vert f\Vert}^2,\ \ \ \forall f\in\mathcal{H}. 
\end{equation*}

An operator associated with a biframe is presented, which is a special case of multiplier operators \cite{Balazs.p}. In the simplest case, multiplier operators are well-defined when the sequences are Bessel sequences, which gives us the Bessel multipliers. Also, the invertibility of this operator is obtained when the sequences are Riesz bases. But, it is interesting that the biframe operator has these properties for sequences that are not even Bessel sequences. Also, the biframe operator associated with a biframe has properties very close to those of the frame operator. Therefore, the reconstruction of the elements of the Hilbert space, which is one of the important achievements of frame theory, is well-done. Due to the important role played by orthonormal bases in a Hilbert space, we are looking for sequences that form a biframe together with an orthonormal basis. This research leads us to sets that are actually bases. The new set of all these bases has a place between the set of all orthonormal bases and the set of all Riesz bases, but studying this topic in pair frames does not have such an achievement for us.

This paper is organized as follows. Section 2 contains some preliminary results and notations that are used throughout the paper. In Section 3, we will introduce the new concept of biframe for a Hilbert space, created by8 a pair of sequences, and we present several examples of biframes. In Section 4, by defining the biframe operator associated with a biframe, we discuss the properties of biframes from the perspective of operator theory. Moreover, we investigate those operators which preserve the biframe property. In Section 5, we classify the biframes whose constituent sequences of them are Bessel sequences, frames, and Riesz bases. Finally, in Section 6, we classify the biframes for which one of the constituent sequences is an orthonormal basis. This classification leads us to new sets, which are partitions of the set of all sequences with this property (namely, the property of forming a biframe with an orthonormal basis). We obtain new bases which are closely related to orthonormal bases and Riesz bases.
\section{Notation and preliminaries}
  \ \ Throughout this paper, $\mathcal{H}$ denotes a separable Hilbert space.
The notation $B(\mathcal{H},\mathcal{K})$ denotes the set of all bounded linear operators from $\mathcal{H}$ to the Hilbert space $\mathcal{K}$. If $\mathcal{H}=\mathcal{K}$, this set is denoted by $B(\mathcal{H})$ and $\mathcal{I}$ denote the identity operator on $\mathcal{H}.$
Also, $GL(\mathcal{H})$ is defined as the set of all invertible, bounded linear operators on $\mathcal{H}$, and ${GL}^{+}(\mathcal{H})$ denotes the subset of $GL(\mathcal{H})$ that consists of positive operators. Finally, ${B}^{+}_{b.b.}(\mathcal{H})$ is defined as the set o2f all self-adjoint and positive operators on $\mathcal{H}$ which are bounded below.

A sequence $\{f_{k}\}_{k=1}^{\infty}$ is called a \emph{frame} for $\mathcal{H}$, if there exist positive constants $A$ and $B$ such that
$$A{\Vert f \Vert}^2\leq\sum_{k=1}^{\infty} {\vert \langle f,f_{k}\rangle\vert}^{2}\leq B{\Vert f\Vert}^2,\ \forall f\in \mathcal{H}.$$
The constants $A$ and $B$ are known as lower and upper frame bounds, respectively. If $A=B$, then the frame $\{f_{k}\}_{k=1}^{\infty}$ is called a \emph{tight frame}; it is called a \emph{Parseval frame} if $A=B=1$. The sequence $\{f_{k}\}_{k=1}^{\infty}$ is called a \emph{Bessel sequence} if only the right inequality holds.
an associated operator to frame $F=\{f_{k}\}_{k=1}^{\infty}$ is frame operator, defined by
 $$S_{F}:\mathcal{H}\longrightarrow\mathcal{H};\ {S}_{F}(f)=\sum_{k=1}^{\infty}\langle f,f_{k}\rangle f_{k}.$$
The frame operator $S_F$ is a self-adjoint operator that belongs to ${GL}^{+}(\mathcal{H})$.

Two Bessel sequences $\{f_{k}\}_{k=1}^{\infty}$ and $\{g_{k}\}_{k=1}^{\infty}$ are \emph{dual frames} for $\mathcal{H}$ if one of the following statements holds.
\begin{enumerate}
\item[(i)] $f=\sum_{k=1}^\infty\langle f,g_k\rangle f_k, \ \ \forall f\in\mathcal{H}$.
\item[(ii)] $f=\sum_{k=1}^\infty\langle f,f_k\rangle g_k, \ \ \forall f\in\mathcal{H}$.
\item[(iii)] $\langle f,g\rangle=\sum_{k=1}^\infty\langle f,f_k\rangle \langle g_k,g\rangle, \ \ \forall f,g\in\mathcal{H}$.
\end{enumerate}

A \emph{Riesz basis}  for $\mathcal{H}$ is a family of the form $\{Ue_{k}\}_{k=1}^{\infty}$, where $\{e_k\}_{k=1}^{\infty}$ is an orthonormal basis for $\mathcal{H}$ and the operator $U$ belongs to $GL(\mathcal{H})$. 

Two sequences $\{f_{k}\}_{k=1}^{\infty}$ and $\{g_{k}\}_{k=1}^{\infty}$ in $\mathcal{H}$ are \emph{biorthogonal} if $\langle f_k,g_j\rangle=\delta_{k,j}$.
We refer the readers to \cite{chris} for more details on frames and bases.

 For $U\in GL(\mathcal{H})$, a sequence $\{f_{k}\}_{k=1}^{\infty}$ in $\mathcal{H}$ is a \emph{U-controlled frame}  \cite{Balazs}, if there exist positive constants $A$ and $B$ such that 
\begin{equation*}
A{\Vert f \Vert}^2\leq\sum_{k=1}^{\infty}  \langle f,f_{k}\rangle\langle Uf_{k} , f\rangle \leq B{\Vert f\Vert}^2,\ \forall f\in \mathcal{H}.
\end{equation*}  

For $U$ and $T$ in $GL(\mathcal{H})$, a sequence $\{f_{k}\}_{k=1}^{\infty}$ in $\mathcal{H}$ is a \emph{(T,U)-controlled frame} \cite{Musa}, if there exist positive constants $A$ and $B$ such that
\begin{equation*}
A{\Vert f \Vert}^2\leq\sum_{k=1}^{\infty}  \langle f,Tf_{k}\rangle\langle Uf_{k} , f\rangle \leq B{\Vert f\Vert}^2,\ \forall f\in \mathcal{H}.
\end{equation*} 

A pair $(F,G)=(\{f_{k}\}_{k=1}^{\infty},\{g_{k}\}_{k=1}^{\infty})$ of sequences in $\mathcal{H}$ is a \emph{pair frame} for $\mathcal{H}$, if the pair frame operator $Sf=\sum_{k=1}^\infty \langle f,f_k\rangle g_k$ is well-defined and invertible for every $f\in\mathcal{H}$.
\begin{prop}$\cite{Musa}$\label{4}
Let $S_{1},S_{2} \in {GL}^{+}(\mathcal{H})$. Then, 
$ S_{2}=VS_{1}U^{*}$ if and only if $U=S_{2}^{r}WS_{1}^{-p}$ and $V=S_{2}^{t}TS_{1}^{-q}$, such that $W$ and $T$ are  bounded operators on $\mathcal{H}$ such that $TW^{*}=\textit{I}$, $p,q,r,t\in \mathbb{R}$, and $r+t=1$, $p+q=1$.
\end{prop}
\section{Biframes}
\ \ In this section, we define the main concept of this paper, namely, the notion of biframe. Next, we present some examples to illustrate the concept of biframes and their constituent sequences more.
\begin{defn}\label{7}
A pair $(F,G)=(\{f_{k}\}_{k=1}^\infty , \{g_{k}\}_{k=1}^\infty)$ of sequences in $\mathcal{H}$ is called a \textbf{ biframe} for $\mathcal{H}$ if there exist  positive constants $A$ and $B$ such that
\begin{equation}\label{65}
A{\Vert f \Vert}^2\leq\sum_{k=1}^\infty \langle f,f_{k}\rangle\langle g_{k},f\rangle\leq B{\Vert f\Vert}^2,\ \ \ \forall f\in\mathcal{H}. 
\end{equation}
The numbers $A$ and $B$ are called lower and upper biframe bounds, respectively.\\ 
The biframe $(\{f_{k}\}_{k=1}^\infty , \{g_{k}\}_{k=1}^\infty)$ is called \textbf{Parseval} if $A=B=1$.
 \end{defn}
After defining th8e concept of biframes, its relation to the previously defined notions must be examined. The following remark shows that biframes generalize ordinary frames and controlled frames.
\begin{rem}
According to \emph{Definition} \ref{7}, the following statements are true for a sequence \\$F=\{f_{k}\}_{k=1}^{\infty}$ in $\mathcal{H}$. %\cite{Duffin, Balazs, Musa}.
\begin{enumerate}
\item[(i)] If $(F,F)$ is a biframe for $\mathcal{H}$, then $F$ is a frame for $\mathcal{H}$.
\item[(ii)] If $(F,UF)$ is a biframe for some $U\in GL(\mathcal{H})$, then $F$ is a U-controlled frame for $\mathcal{H}$.
\item[(iii)] If $(TF,UF)$ is a biframe for some $T$ and $U$ in $GL(\mathcal{H})$, then $F$ is a (T,U)-controlled frame for $\mathcal{H}$.
\end{enumerate}
 \end{rem}
A biframe is a pair of sequences. Therefore, the relationship between the two sequences that form a biframe is very important. To investigate the relationship between these two sequences, we try to answer the following questions.
 \begin{enumerate}
 \item[$(Q_1)$] What kind of Bessel sequences make biframes? And how about frames, Riesz bases and orthonormal bases?
 %The construct of one sequence has any effect on the construct of another sequence?
 \item[$(Q_2)$] Are there sequences that are not Bessel sequences, frames, Riesz bases or orthonormal bases, but form a biframe?
 %Is it necessary that every two Bessel sequences, frames or, Riesz bases form a biframe?
 \item[$(Q_3)$] Are the types of sequences that form a biframe interdependent? For example, if one of the sequences is a Bessel sequence, does it necessarily follow that the other one is also a Bessel sequence? What about Riesz bases and orthonormal bases?
 \end{enumerate}
In the following sections, we will answer these questions. But before that, we present some examples of biframes.
 \begin{example}\label{35}
 Let $\{e_{k}\}_{k=1}^{\infty}$ be an orthonormal basis for $\mathcal{H}$.
 \begin{enumerate}
\item[(i)] \textbf{A biframe may be constructed using two non-Bessel sequences.}\\
It is clear that 
\begin{center}
$\{f_{k}\}_{k=1}^{\infty}=\{e_{1},2e_{2},\frac{1}{3}e_{3},4e_{4},\frac{1}{5}e_{5},6e_{6},\ldots\},$\\
$\{g_{k}\}_{k=1}^{\infty}=\{2e_{1},e_{2},4e_{3},\frac{1}{3}e_{4},6e_{5},\frac{1}{5}e_{6}\ldots\}\ $
\end{center}
are not Bessel sequences. But, the pair $(\{f_{k}\}_{k=1}^{\infty},\{g_{k}\}_{k=1}^{\infty})$ is a biframe, because for $f\in\mathcal{H}$ we can write
$${\Vert f\Vert}^{2}\leq\sum_{k=1}^{\infty}\langle f,f_{k}\rangle\langle g_{k},f\rangle=
\sum_{k=1}^{\infty}\frac{2k}{2k-1}({\vert\langle f,e_{2k-1}\rangle\vert}^{2}+{\vert\langle f,e_{2k}\rangle\vert}^{2})
\leq 2{\Vert f\Vert}^{2}.$$
\item[(ii)] \textbf{Two Bessel sequences may not form a biframe.}\\
It is easy to see that
\begin{center}8
$\{f_{k}\}_{k=1}^{\infty}=\{e_1,\frac{1}{2}e_2,\frac{1}{3}e_3,\frac{1}{4}e_4\ldots\},\ \ \ $\\
$\{g_{k}\}_{k=1}^{\infty}=\{0,e_1,0,\frac{1}{2}e_2,0,\frac{1}{3}e_3,\ldots\}$
\end{center}
are Bessel sequences with bound 1. But, the pair $(\{f_{k}\}_{k=1}^{\infty}  ,\{g_{k}\}_{k=1}^{\infty})$ is not a biframe, because for $f\in\mathcal{H}$ we obtain
\begin{align*}
\sum_{k=1}^{\infty}\langle f,f_{k}\rangle\langle g_{k},f\rangle &=
\langle f,e_{1}\rangle\langle 0,f\rangle+\langle f,\frac{1}{2}e_{2}\rangle\langle e_{1},f\rangle\\
&+\langle f,\frac{1}{3}e_3\rangle\langle 0,f\rangle
\langle f, \frac{1}{4}e_4\rangle\langle\frac{1}{2}e_2,f\rangle+\cdots\\
&=\frac{1}{2}\langle f,e_{2}\rangle\langle e_{1},f\rangle+\frac{1}{8}\langle f,e_{4}\rangle\langle e_2,f\rangle+\cdots. 
\end{align*}
Now, if we set $f=e_{1}$, then the summation in the definition of biframe is equal to 0:
$$\sum_{k=1}^{\infty}\langle f,f_{k}\rangle\langle g_{k},f\rangle=0.$$
This shows that the summation does not have a non-zero lower bound.\\
\item[(iii)] \textbf{Two frames may not form a biframe.}\\
Consider the sequences 
\begin{center}
$\{f_{k}\}_{k=1}^{\infty}=\{-\frac{1}{2}e_{1},\frac{1}{2}e_{1},e_{2},e_{3},\ldots\},$\\
$\{g_{k}\}_{k=1}^{\infty}=\{e_{1},e_{1},e_{2},e_{3},\ldots\}.\ \ \ \ \ $
\end{center}
The sequence $\{f_{k}\}_{k=1}^{\infty}$ is a frame with bounds $\frac{1}{2}$ and 1, and the sequence $\{g_{k}\}_{k=1}^{\infty}$ is a frame with bounds 1 and 2. But, these frames cannot form a biframe, because for $f\in\mathcal{H}$ we can write 
\begin{align*}
\sum_{k=1}^{\infty}\langle f,f_{k}\rangle\langle g_{k},f\rangle &=
\langle f,-\frac{1}{2}e_{1}\rangle\langle e_{1},f\rangle+\langle f,\frac{1}{2}e_{1}\rangle\langle e_{1},f\rangle\\
&+\langle f,e_{2}\rangle\langle e_{2},f\rangle+\langle f,e_{3}\rangle\langle e_{3},f\rangle+\cdots\\
&=\langle f,e_{2}\rangle\langle e_{2},f\rangle+\langle f,e_{3}\rangle\langle e_{3},f\rangle+\cdots\\ &=\sum_{k=2}^\infty{\vert\langle f,e_k\rangle\vert}^2.
\end{align*}
Now, $\sum_{k=1}^{\infty}\langle f,f_{k}\rangle\langle g_{k},f\rangle=0$ if we set $f=e_{1}$. This implies that the summation does not have a non-zero lower bound.
Therefore, the pair $(\{f_{k}\}_{k=1}^{\infty},\{g_{k}\}_{k=1}^{\infty})$ is not a biframe.
\item[(iv)] \textbf{Two Riesz bases may not form a biframe.}\\
We cosider the orthonormal basis $\{e_k(x)\}_{k\in\mathbb{Z}}=\{e^{2\pi ikx}\}_{k\in\mathbb{Z}}$ for $L^2(0,1)$.\\ $T_{-1}$ and $E_1$ are translation and modulation operators on $L^2(0,1)$ respectively. Because these operators are unitary, the following sequences are Riesz bases for $L^2(0,1)$. 
$$\{f_{k}\}_{k\in\mathbb{Z}}=\{T_{-1}(e_k(x))\}_{k\in\mathbb{Z}}=\{e^{2\pi ik(x+1)}\}_{k\in\mathbb{Z}},$$
$$\{g_{k}\}_{k\in\mathbb{Z}}=\{E_{1}(e_k(x))\}_{k\in\mathbb{Z}}=\{e^{2\pi i(k+1)x}\}_{k\in\mathbb{Z}}.$$
For $f\in L^2(0,1)$, we have
\begin{align*}
\sum_{k\in\mathbb{Z}}\langle f,f_k\rangle\langle g_k,f\rangle
&=\sum_{k\in\mathbb{Z}}\Big\langle f,e^{2\pi ik(x+1)}\Big\rangle\Big\langle e^{2\pi i(k+1)x},f\Big\rangle\\
&=\sum_{k\in\mathbb{Z}}\Big(\int_{0}^{1} f(x)e^{-2\pi ik(x+1)}dx\Big)\Big(\int_{0}^{1} e^{2\pi i(k+1)x} \overline{f(x)}dx\Big).
\end{align*}
Now, set $f=e^{2\pi ix},
thbb{Z}\Big(\int_{0}^{1} e^{2\pi i(1-k)x}dx\Big)\Big(\frac{1}{2\pi ik}(e^{2k\pi i}-e^{0})\Big)=0.$
$$\sum_{k\in\mathbb{Z}}\langle e^{2\pi ix} ,e^{2\pi ik(x+1)}\rangle\langle e^{2\pi i(k+1)x},e^{2\pi ix}\rangle
=\sum_{k\in\mathbb{Z}}\Big(\int_{0}^{1} e^{2\pi i(1-k)x}dx\Big)\Big(\int_{0}^{1} e^{2\pi ikx}\Big)=0.$$
Therefore, the summation does not have a non-zero lower bound, so the pair $(\{f_{k}\}_{k\in\mathbb{Z}},\{g_{k}\}_{k\in\mathbb{Z}})$ isnot a biframe for $L^2(0,1).$
\item[(v)] \textbf{Two orthonormal bases may not from a biframe.}\\
Let $\sigma$ be a permutation of $\mathbb{N}$ other than the identity. Then, the pair \\ $(\{e_{k}\}_{k=1}^{\infty},\{e_{\sigma(k)}\}_{k=1}^{\infty})$ is not a biframe. In fact, an easy calculation shows that the summation in the definition of biframe is always equal to 0 for such sequences. If we set $f=e_j$ for some $j\in\mathbb{N}$ such that $\sigma (j)\neq j$, then
$$\sum_{k=1}^{\infty}\langle f,e_{k}\rangle\langle e_{\sigma(k)},f\rangle=\sum_{k=1}^{\infty}\langle e_j,e_{k}\rangle\langle e_{\sigma(k)},e_j\rangle=0.$$
\end{enumerate}
\end{example}
A pair frame is a concept related to pair $(F,G)$ of sequences in $\mathcal{H}$. Although, a biframe is introduced as a pair of sequences with the same symbol, but these two notions are not actually the same. In fact, biframes are special types of pair frames, as we  show in \emph{Corollary} \ref{44}. In Section 6, we focus on the differences between these concepts, where we specifically consider one of their constituent sequences as the orthonormal basis.

The next example shows that a pair frame is not necessarily a biframe.
\begin{example}\label{37}
Consider the Hilbert space $\mathbb{R}^2$, and 
the sequences $F=\{f_{k}\}_{k=1}^{2}$ and $G=\{g_{k}\}_{k=1}^{2}$ defined as follows. 
\begin{center}
$\{f_{k}\}_{k=1}^{2}=\{(1,2),(\frac{8}{7},4)\},\ \ \ \ \ \ \ $\\
$\{g_{k}\}_{k=1}^{2}=\{(-1,\frac{175}{2}),(\frac{7}{4},\frac{-14}{3})\}.$
\end{center}
First, we show that $(F,G)$ is a pair frame by considering the properties of the pair frame operator $S$.
For $(x,y)\in\mathbb{R}^2$, the definition of $S$ is given by the following equalities.
\begin{align*}
S(x,y) &=\big\langle (x,y),(1,2)\big\rangle (-1,\frac{175}{2})+\big\langle (x,y),(\frac{8}{7},4)\big\rangle(\frac{7}{4},\frac{-14}{3})\\
&=(x+2y)(-1,\frac{175}{21})+(\frac{8}{7}x+4y)(\frac{7}{4},\frac{-14}{3})\\
&=(-(x+2y),\frac{175}{21}x+\frac{350}{21}y)+(2x+7y,\frac{-112}{21}x-\frac{56}{3}y)\\
&=(x+5y,3x-2y).
\end{align*}
The matrix associated with the operator $S$ is 
%\begin{center}
$$[S]=\begin{bmatrix}
1& 5\\
3& -2\\
\end{bmatrix}.$$
It is clear that this matrix is invertible ($det (S)=-17\neq 0$). So, the operator $S$ is well-defined and invertible and therefore, $(F,G)$ is a pair frame. But, this pair is not a biframe. For $(x,y)\in\mathbb{R}^2$,
$$\big\langle (x,y),(1,2)\big\rangle \big\langle(-1,\frac{175}{2}),(x,y)\big\rangle
+\big\langle (x,y),(\frac{8}{7},4)\big\rangle\big\langle(\frac{7}{4},\frac{-14}{3}),(x,y)\big\rangle= x^2-y^2+8xy.$$
Now, if we set $(x,y)=(-1+\frac{\sqrt{6}}{2},\frac{1}{\sqrt{2}})$ then $ x^2-y^2+8xy=0$. Hence this phrase has no nonzero lower bound.
\end{example}
 
 %-------------------------------------------------------------------------------------------------------------------------------------
 \section{The biframe operator}
 \ \ Since finding the bounds of a biframe is not always easy in practice, to better understand and be able to work with biframes, we need to introduce an operator similar to the frame operator, one that has as many good and useful properties as the frame operator. In this section, we introduce the biframe operator associated with a biframe, and we examine its properties. Also, we characterize biframes based on the properties of this operator.
\begin{defn}\label{66}
Let $(F,G)=(\{f_{k}\}_{k=1}^{\infty},\{g_{k}\}_{k=1}^{\infty})$ be a biframe for $\mathcal{H}$. The biframe operator ${S}_{F,G}$ is defined by
\begin{equation}\label{8}
S_{F,G}:\mathcal{H}\longrightarrow\mathcal{H},\ \ 
 S_{F,G}(f):=\sum_{k=1}^{\infty}\langle f,f_{k}\rangle g_{k}.
\end{equation}
\end{defn}
In what follows, we present some properties of the biframe operator.
\begin{theorem}\label{31}
Let $(F,G)=(\{f_{k}\}_{k=1}^{\infty} , \{g_{k}\}_{k=1}^{\infty})$ be a biframe for $\mathcal{H}$ with bounds A and B. Then, the following statements are true.
\begin{enumerate}
\item[(i)] The operator $S_{F,G}$ is well-defined, bounded, positive, and invertible. 
\item[(ii)] $(F,G)$ is a biframe if and only if $(G,F)$ is a biframe.
\end{enumerate} 
\end{theorem}
\begin{proof}
%\begin{enumerate}
(i) To prove that ${S}_{F,G}$ is well-defined, let $f\in\mathcal{H}$. For $n\in\mathbb{N}$, define $${S}_{n}f=\sum_{k=1}^{n}\langle f,f_{k} \rangle g_{k}.$$ The sequence $\{S_{n}\}_{n=1}^{\infty}$ is a sequence of linear and bounded operators on $\mathcal{H}$. Consider $i,j\in\mathbb{N}, i>j$. Then, 
\begin{align*}
\vert \langle {S}_{i}f,f\rangle - \langle {S}_{j}f,f\rangle\vert&=\vert \langle ({S}_{i}f-{S}_{j}f),f\rangle\vert\\
&=\vert \langle \sum_{k=j+1}^{i}\langle f,f_{k}\rangle g_{k},f\rangle\vert\\
&\leq\sum_{k=j+1}^{i}\vert\langle f,f_{k}\rangle \langle g_{k},f\rangle\vert.
\end{align*}
By \emph{Definition} \ref{7}, the series $\sum_{k=1}^{\infty}\langle f,f_{k}\rangle \langle g_{k},f\rangle $ converges in $\mathbb{R}$. So, its associated sequence of partial sums is a Cauchy sequence. Thus, the sequence $\{\langle {S}_{n}f,f\rangle\}_{n=1}^{\infty}$ is Cauchy, and therefore an operator $S\in B(\mathcal{H})$ exists (see \cite{Kub}, for example) such that $\{S_{n}\}_{n=1}^{\infty}$ converges to $S$ weakly:
$$\langle Sf,f\rangle=\lim_{n\to\infty}\langle {S}_{n}f,f\rangle=\langle\sum_{k=1}^{\infty}\langle f,f_{k}\rangle g_{k},f\rangle\\=\langle{S}_{F,G}f,f\rangle.$$
By our definition of the biframe operator $S_{F,G}$ and the uniqueness of limit, we conclude that ${S}_{F,G}=S$, and therefore ${S}_{F,G}$ is a well-defined and bounded operator.
By our definition of ${S}_{F,G}$, for every $f\in\mathcal{H}$ we obtain $\sum_{k=1}^{\infty}\langle f,f_{k}\rangle \langle g_{k},f\rangle=\langle{S}_{F,G}f,f\rangle.$
Now, by \emph{Definition} \ref{7}, we simply observe that
\begin{equation}\label{9}
A{\Vert f\Vert}^{2}\leq\langle S_{F,G}f,f\rangle\leq B{\Vert f\Vert}^{2}.
\end{equation}
Which implies that $S_{F,G}$ is a positive operator.\\
To prove that $S_{F,G}$ is an invertible operator, we need to show that $S_{F,G}$ and $S_{F,G}^*$ are injective and have closed ranges \cite{safapour}. \\
For $f,g\in\mathcal{H}$,
\begin{equation*}\label{38}
\langle S_{F,G}f,g\rangle =\sum_{k=1}^{\infty}\langle f,f_k\rangle \langle g_k,g\rangle =\langle f,S_{G,F}\rangle .
\end{equation*}
Hence, $S_{F,G}^*=S_{G,F}$. By the definition of biframe, $S_{F,G}$ and $S_{G,F}$ are injective.\\
%Let $f\in\mathcal{N}(S_{F,G})$, then $S_{F,G} f=0$, so $\langle S_{F,G} , f \rangle=0$.\\
 %By \emph{Definition \ref{7}.}, we have $\Vert f \Vert =0$, therefore $f=0$ and so $S_{F,G}$ is injective.\\
To prove that $S_{F,G}$ has a closed range, 
%$$\overline{{\mathcal{R}(S_{F,G})}}={\mathcal{N}({S_{F,G}}^{*})}^{\perp}={\mathcal{N}(S_{F,G})}^{\perp}={\{0\}}^{\perp}=\mathcal{H},$$
%it is enough that show $\mathcal{R}(S_{F,G})$ is closed.\\
let $\{h_{n}\}_{n=1}^{\infty}\subset\mathcal{R}(S_{F,G}) $ be a sequence that converges to $h\in\mathcal{H}$.
Then, there exists a sequence $\{t_{n}\}_{n=1}^{\infty}$ in $\mathcal{H}$ such that $S_{F,G}(t_{n})=h_{n} $ for every $n$; this means that the sequence $\{S_{F,G}(t_{n})\}_{n=1}^{\infty}$ converges to $h$.\\
The sequence $\{h_{n}\}$ is convergent, so it is a Cauchy sequence. Let ${\epsilon>0}$ be given.\\
 Then, there exists $N>0$ such that 
 $$\Vert h_{n}-h_{m}\Vert=\Vert S_{F,G}(t_{n})-S_{F,G}(t_{m})\Vert\leq \epsilon,\ \forall m,n\geq N.$$
 Consider $m,n\geq N,\ n>m$. Now, by (\ref{9}) and the \emph{Cauchy-Schwarz inequality}, 
\begin{align*}
A{\Vert t_{n}-t_{m}\Vert}^{2}&\leq \langle S_{F,G}(t_{n}-t_{m}) , (t_{n}-t_{m})\rangle\\
&\leq\Vert S_{F,G}(t_{n})-S_{F,G}(t_{m})\Vert\Vert t_{n}-t_{m}\Vert\\
&\leq\epsilon\Vert t_{n}-t_{m}\Vert.
\end{align*}
Hence $\Vert t_{n}-t_{m}\Vert\leq\frac{\epsilon}{A}$, which implies that the sequence $\{t_{n}\}_{n=1}^{\infty}$ is a Cauchy sequence in $\mathcal{H}$, and so it converges to some $t\in\mathcal{H}$.
Since $S_{F,G}$ is bounded,
$$ S_{F,G} (t_{n})\longrightarrow S_{F,G}( t ),\ as\ n\longrightarrow \infty.$$
On the other hand,
$$S_{F,G} (t_{n})\longrightarrow h,\ as\ n\longrightarrow \infty.$$
Now by the uniqueness of limit we obtain $S_{F,G}( t)=h$. So, $h\in\mathcal{R}(S_{F,G})$ and therefore, $\mathcal{R}(S_{F,G})$ is closed. Thus, $\mathcal{R}(S_{F,G}^*)$ is closed.\\
(ii) Let $(F,G)$ be a  biframe with bounds $A$ and $B$. Then for every $f\in\mathcal{H}$, 
$$A{\Vert f\Vert}^2\leq\sum_{k=1}^{\infty}\langle f,f_k\rangle\langle g_k,f\rangle\leq B{\Vert f\Vert}^2.$$
Now, we can write
$$\sum_{k=1}^{\infty}\langle f,f_k\rangle\langle g_k,f\rangle =\overline{\sum_{k=1}^{\infty}\langle f,f_k\rangle \langle g_k,f\rangle}=\sum_{k=1}^{\infty}\langle f,g_k\rangle \langle f_k,f\rangle.$$
This implies that 
$$A{\Vert f\Vert}^2\leq\sum_{k=1}^{\infty}\langle f,g_k\rangle\langle f_k,f\rangle\leq B{\Vert f\Vert}^2.$$
So, $(G,F)$ is a biframe with bounds $A$ and $B$. The converse of this statement can be proved similarly.
\end{proof}
%-----------------------------------------------------------------------------------------------------------------------------------------
%------------------------------------------------------------------------------------------------------------------------------------------
\begin{cor}\label{44}
Every biframe $(F,G)$ is a pair frame.
\end{cor}
\begin{proof}
Let $(F,G)$ be a biframe. By \emph{Theorem} \ref{31}, $S_{F,G}$ is well-defined and invertible. So, by the definition of pair frames, $(F,G)$ is a pair frame.
\end{proof}
Reconstruction of the elements of a Hilbert space from the frame coefficients is one of the most important
 achievements of frame theory. The next theorem deals with this reconstruction by using biframes.
\begin{theorem}\label{43}
Suppose that $(\{f_{k}\}_{k=1}^{\infty},\{g_{k}\}_{k=1}^{\infty})$ is a biframe for $\mathcal{H}$ with biframe operator $S_{F,G}$. Then, for every $f\in\mathcal{H}$ the following reconstruction formula holds.
\begin{equation}\label{10}
 f=\sum_{k=1}^{\infty}\langle f,S^{-1}_{G,F}f_{k}\rangle g_{k}=\sum_{k=1}^{\infty}\langle f,f_{k}\rangle S^{-1}_{F,G}g_{k}. 
\end{equation}
Moreover, in a complex Hilbert space, we can rewrite the first equation of (\ref{10}) in the following form.
\begin{equation}\label{68}
f=\sum_{k=1}^{\infty}\langle f,S^{-1}_{G,F}f_{k}\rangle g_{k}=\sum_{k=1}^{\infty}\langle f,S^{-1}_{F,G}f_{k}\rangle g_{k}.
\end{equation}
\end{theorem}
\begin{proof}
We observed in the proof of \emph{Theorem} \ref{31} that ${S_{F,G}}^*=S_{G,F}$. Let $f\in \mathcal{H}.$ Then
\begin{equation*}
f=S_{F,G}S^{-1}_{F,G}f=\sum_{k=1}^{\infty}\langle S^{-1}_{F,G}f,f_{k}\rangle g_{k}=\sum_{k=1}^{\infty}\langle f,S^{-1}_{G,F}f_{k}\rangle g_{k},
\end{equation*}
and
\begin{equation*}
f=S^{-1}_{F,G}S_{F,G}f=S^{-1}_{F,G}\sum_{k=1}^{\infty}\langle f,f_{k}\rangle g_{k}=\sum_{k=1}^{\infty}\langle f,f_{k}\rangle S^{-1}_{F,G}g_{k}.
\end{equation*}
For the final assertion note that in a complex Hilbert space, every positive operator is self-adjoint. (The next example shows that this is not true in real Hilbert spaces.) So, we have the self-adjointness property for the positive operator $S_{F,G}$ in complex Hilbert spaces and therefore, the result follows clearly by using $S_{F,G}=S_{G,F}$ in  (\ref{10}).
\end{proof}
The biframe decomposition stated in (\ref{10}) is the \textit{reconstruction formula} of biframe theory. The sequence $\{\langle f,S^{-1}_{F,G}f_{k}\rangle\}_{k=1}^\infty$ is called the \textit{sequence of biframe coefficients}.
\begin{example}\label{42}
Consider the space $\mathbb{R}^2$ as a real Hilbert space.\\
 The matrix 
 $A=\begin{bmatrix}
1 & 3\\
2 & 8 
\end{bmatrix}$
is positive because
for $\begin{bmatrix}
x\\
y
\end{bmatrix} \in\mathbb{R}^2,$
$$\Big\langle\begin{bmatrix}
1 & 3\\ 
2 & 8
\end{bmatrix}
\begin{bmatrix}
x\\
y
\end{bmatrix},
\begin{bmatrix}
x\\
y
\end{bmatrix}\Big\rangle=\Big\langle
\begin{bmatrix}
x+3y\\
2x+8y
\end{bmatrix},
\begin{bmatrix}
x\\
y
\end{bmatrix}\Big\rangle=x^2+5xy+8y^2=(x+\frac{5}{2}y)^2+\frac{7}{4}y^2>0.$$
Clearly, the adjoint of this matrix is 
$A=\begin{bmatrix}
1 & 2\\
3 & 8
\end{bmatrix},$
and this shows that $A$ is not self-adjoint.
\end{example}
As mentioned at the beginning of this section, proving that two sequences form a biframe by using the definition is by no means an easy task. Therefore, we try to characterize biframes based on the properties of the biframe operator. The next theorem presents these conditions.
\begin{theorem}\label{32}
Let $\{f_{k}\}_{k=1}^{\infty}$ and $\{g_{k}\}_{k=1}^{\infty}$ be sequences in a complex Hilbert space $\mathcal{H}$. Then,\\
$(\{f_{k}\}_{k=1}^{\infty},\{g_{k}\}_{k=1}^{\infty})$ is a biframe for $\mathcal{H}$ if and only if $S_{F,G}$ is a positive and bounded below operator.
\end{theorem}
\begin{proof}
Let $(\{f_{k}\}_{k=1}^{\infty},\{g_{k}\}_{k=1}^{\infty})$ be a biframe for $\mathcal{H}$. As we observed in \emph{Theorem} \ref{31}, $S_{F,G}$ is a positive and invertible operator. Therefore, it is bounded below, too.
Conversely, suppose that $S_{F,G}$ is a positive and bounded below operator on $\mathcal{H}$.
Since $S_{F,G}$ is a positive operator, it is self-adjoint. For $f\in\mathcal{H}$,
$$\sum_{k=1}^{\infty}\langle f,f_{k}\rangle \langle g_{k},f\rangle =\langle S_{F,G}f, f\rangle=\vert\langle S_{F,G}f, f\rangle\vert\leq\Vert S_{F,G}\Vert{\Vert f\Vert}^{2}.$$
To obtain the lower bound, by the positivity of  $S_{F,G}$ \cite{Kub}, we obtain the following inequality.
\begin{equation}\label{11}
{\Vert S_{F,G}f\Vert}^{2}\leq\Vert S_{F,G}\Vert\langle S_{F,G}f,f\rangle.
\end{equation}
On the other hand,  $S_{F,G}$ is bounded below with lower bound $\alpha$:
\begin{equation}\label{12}
\exists \alpha >0,\ \ \ \   \alpha\Vert f\Vert\leq\Vert S_{F,G}f\Vert.
\end{equation}
Using (\ref{11}) and (\ref{12}) we find that 
$${\alpha}^{2}{\Vert f\Vert}^{2}\leq\Vert S_{F,G}\Vert\langle S_{F,G}f,f\rangle, \ \ \  \forall f\in\mathcal{H}.$$
This implies that $\langle S_{F,G}f,f \rangle\geq\frac{{\alpha}^{2}}{\Vert S_{F,G}\Vert} {\Vert f\Vert}^{2}$ and so, $\frac{{\alpha}^{2}}{\Vert S_{F,G}\Vert}$ is a lower bound for $\langle S_{F,G}f,f \rangle$. Thus, 
 $(\{f_{k}\}_{k=1}^{\infty},\{g_{k}\}_{k=1}^{\infty})$ is a biframe for $\mathcal{H}$.
\end{proof}
To answer the questions we posed in the previous section about the structural relationship between the two sequences that form a biframe, we establish the next theorem which determines the dependency of two Riesz bases that form a biframe. First, we show by some examples that this dependency is not necessarily true for Bessel sequences and frames.
\begin{example}\label{41}
Let $\{e_k\}_{k=1}^\infty$ be an orthonormal basis for $\mathcal{H}$.  %The sequences $\{f_k\}_{k=1}^\infty$ and $\{g_k\}_{k=1}^\infty$ are considered as the following.\\
\begin{enumerate}
\item[(i)] \textbf{A biframe constructed by a frame and a non-frame}\\
 Consider the following sequences.
\begin{center}
$\{f_k\}_{k=1}^\infty=\{e_1,e_1,e_1,e_2,e_2,e_2,e_3,e_3,e_3,\ldots\},\ \ \ \ \ \ \ \ \ \ \ $\\
$\{g_k\}_{k=1}^\infty=\{2e_1,e_1,-e_1,\frac{3}{2}e_2,e_1,-e_1,\frac{4}{3}e_3,e_1,-e_1,\ldots\}$.
\end{center}
For $f\in\mathcal{H}$,
$$\sum_{k=1}^\infty\langle f,f_k\rangle\langle g_k,f\rangle=\sum_{k=1}^\infty\frac{k+1}{k}{\vert\langle f,e_k\rangle\vert}^2.$$
Also,
$${\Vert f\Vert}^2\leq\sum_{k=1}^\infty\frac{k+1}{k}{\vert\langle f,e_k\rangle\vert}^2\leq 2{\Vert f\Vert}^2.$$
Therefore, $(\{f_k\}_{k=1}^\infty,\{g_k\}_{k=1}^\infty)$ is a biframe with bounds 1 and 2, although $\{f_k\}_{k=1}^\infty$ is a frame and $\{g_k\}_{k=1}^\infty$ is not.
\item[(ii)] \textbf{A Parseval biframe constructed by a Bessel sequence and a non-Bessel sequence}\\
Consider the following sequences.
\begin{center}
$\{f_k\}_{k=1}^\infty=\{e_1,\frac{1}{2}e_2,e_3,\frac{1}{4}e_4,e_5,\ldots\},$\\
$\{g_k\}_{k=1}^\infty=\{e_1,2e_2,e_3,4e_4,e_5,\ldots\}.$
\end{center}
For $f\in\mathcal{H}$, we obtain the following equalities.
$$\sum_{k=1}^\infty\langle f,f_k\rangle\langle g_k,f\rangle=\sum_{k=1}^\infty{\vert\langle f,e_k\rangle\vert}^2={\Vert f\Vert}^2.$$
\end{enumerate}
Hence, $(\{f_k\}_{k=1}^\infty,\{g_k\}_{k=1}^\infty)$ is a Parseval biframe, although $\{f_k\}_{k=1}^\infty$ is a Bessel sequence and $\{g_k\}_{k=1}^\infty$ is not.
\end{example}
\begin{theorem}\label{39}
Suppose that $(\{f_{k}\}_{k=1}^{\infty},\{g_{k}\}_{k=1}^{\infty})$ is a biframe for $\mathcal{H}$. Then, $\{f_{k}\}_{k=1}^{\infty}$ is a Riesz basis for $\mathcal{H}$ if and only if $\{g_{k}\}_{k=1}^{\infty}$ is a Riesz basis for $\mathcal{H}$.
\end{theorem}
\begin{proof}
Let $\{f_{k}\}_{k=1}^{\infty}$ be a Riesz basis. Then, there exists an operator $V\in GL(\mathcal{H})$ such that $f_{k=}Ve_{k}$, for all $k\in\mathbb{N}$, where $\{e_{k}\}_{k=1}^{\infty}$ is an orthonormal basis for $\mathcal{H}$.\\
Define the operator
%\begin{eqnarray*}
$$U:\mathcal{H}\rightarrow\mathcal{H},\ \ Uf=S_{F,G}{({V}^{*})}^{-1}f.$$
%\end{eqnarray*}
Clearly, $U\in GL(\mathcal{H})$ and for $k\in\mathbb{N}$, we obtain
\begin{align*}
Ue_{k}=S_{F,G}{({V}^{*})}^{-1}e_{k}&=\sum_{i=1}^{\infty}\langle {({V}^{*})}^{-1}e_{k},f_{i}\rangle g_{i}\\
&=\sum_{i=1}^{\infty}\langle {({V}^{*})}^{-1}e_{k},Ve_{i}\rangle g_{i}\\
&=\sum_{i=1}^{\infty}\langle {V}^{*}{({V}^{*})}^{-1}e_{k},e_{i}\rangle g_{i}\\
&=\sum_{i=1}^{\infty}\langle e_{k},e_{i}\rangle g_{i}=g_{k}.
\end{align*}
Hence $ g_{k}=Ue_{k}$, for all $k\in\mathbb{N}$. This means that $\{g_{k}\}_{k=1}^{\infty}$ is a Riesz basis for $\mathcal{H}$.\\
Similarly, we can prove that  $\{f_{k}\}_{k=1}^{\infty}$ is a Riesz basis if $\{g_{k}\}_{k=1}^{\infty}$ is a Riesz basis.
\end{proof}
\begin{rem}\label{67}
As mentioned before, a positive operator in a complex Hilbert space is self-adjoint. To use the self-adjointness property of the positive biframe operator $S_{F,G}$, from now on we assume that the Hilbert space $\mathcal{H}$ is complex.
\end{rem}
To investigate the operators that preserve the biframe property, since we deal with a pair of sequences in the definition of a biframe, we can study the action of two different operators on the sequences in a biframe. The next theorem presents some operators that preserve the biframe property of a given biframe.
\begin{theorem}\label{29}
Suppose that $(F,G)=(\{f_{k}\}_{k=1}^{\infty},\{g_{k}\}_{k=1}^{\infty})$ is a biframe for $\mathcal{H}$ with biframe\\ operator $S_{F,G}$. Then, the following statements are true.
\begin{enumerate}
\item[(i)]$(\{Uf_{k}\}_{k=1}^{\infty},\{Vg_{k}\}_{k=1}^{\infty})$ is a biframe for $\mathcal{H}$, for some operators $U$ and $V$ in $B(\mathcal{H})$ if and only if there exist operators $Q\in{B}^{+}_{b.b.}(\mathcal{H})$ and   $T,W\in B(\mathcal{H})$ such that $T{W}^{*}=\textit{I}$ and\\
 $U={Q}^{r}W{S}^{-p}_{F,G}\ ,\ V={Q}^{t}T{S}^{-q}_{F,G},$ when $p,q,r,t\in\mathbb{R}$ such that $p+q=1$ and $r+t=1$. 
 \item[(ii)]$(\{Uf_{k}\}_{k=1}^{\infty},\{Vg_{k}\}_{k=1}^{\infty})$ is a Parseval biframe for $\mathcal{H}$, for some operators $U$ and $V$ in $B(\mathcal{H})$ if and only if there exist operators $T,W\in B(\mathcal{H})$ such that $T{W}^{*}=\textit{I}$ and $U=W{S}^{-p}_{F,G}\ ,\ V=T{S}^{-q}_{F,G},$ when $p,q\in\mathbb{R}$ such that $p+q=1$.  
 \end{enumerate}
In particular, if $(\{f_{k}\}_{k=1}^{\infty},\{g_{k}\}_{k=1}^{\infty})$ is a Parseval biframe, then 
 $(\{Uf_{k}\}_{k=1}^{\infty},\{Vg_{k}\}_{k=1}^{\infty})$ is a Parseval biframe if and only if  $VU^{*}=\textit{I}.$
\end{theorem}
\begin{proof}
(i) Suppose that $(UF,VG)=(\{Uf_{k}\}_{k=1}^{\infty},\{Vg_{k}\}_{k=1}^{\infty})$ is a biframe for $\mathcal{H}$ with biframe operator $S_{UF,VG}$, when $U,V\in B(\mathcal{H}).$ For $f\in{\mathcal{H}}$,
\begin{equation}\label{13}
S_{UF,VG}f=\sum_{k=1}^{\infty}\langle f,Uf_{k}\rangle Vg_{k}=V\sum_{k=1}^{\infty}\langle {U}^{*} f,f_{k}\rangle g_{k}
=VS_{F,G}{U}^{*}f.
\end{equation}
Since $(\{f_{k}\}_{k=1}^{\infty},\{g_{k}\}_{k=1}^{\infty})$ and $(\{Uf_{k}\}_{k=1}^{\infty},\{Vg_{k}\}_{k=1}^{\infty})$ are biframes, \emph{Theorem} \ref{32} allows us to conclude that the operators $S_{F,G}$ and $S_{UF,VG}$ are in ${B}^{+}_{b.b.}(\mathcal{H})$, and also are in ${GL}^{+}(\mathcal{H}).$
By (\ref{13}) and \emph{Proposition} \ref{4}, there exist operators $ T,W\in B(\mathcal{H})$ such that $T{W}^{*}=\textit{I}$ and
\begin{equation}\label{14}
U={S}^{r}_{UF,VG}W{S}^{-p}_{F,G}\ \ \  ,\ \ \  V={S}^{t}_{UF,VG}T{S}^{-q}_{F,G}.
\end{equation} 
Conversely, let $Q\in {B}^{+}_{b.b.}(\mathcal{H})$, and let $T$ and $W$ be operators in $B(\mathcal{H})$ such that $T{W}^{*}=\textit{I}$ and
\begin{equation}\label{15}
U={Q}^{r}W{S}^{-p}_{F,G}\ \ \  ,\ \ \  V={Q}^{t}T{S}^{-q}_{F,G}.
\end{equation}
Then, by \emph{Proposition} \ref{4}, $Q=VS_{F,G}{U}^{*}$, which means that for every $f\in\mathcal{H}$,
$$Qf=VS_{F,G}{U}^{*}f
=V(\sum_{k=1}^{\infty}\langle {U}^{*}f,f_{k}\rangle g_{k})
=\sum_{k=1}^{\infty}\langle f,U f_{k}\rangle Vg_{k}.$$
Therefore, Q is of the form of a biframe operator. By the assumption $Q\in{B}^{+}_{b.b.}(\mathcal{H})$,  \emph{Theorem} \ref{32} allows us to conclude that $(\{Uf_{k}\}_{k=1}^{\infty},\{Vg_{k}\}_{k=1}^{\infty})$ is a biframe for $\mathcal{H}.$\\
(ii) This part is a direct consequence of the previous part. 
\end{proof}

Having in mind the final part of \emph{Theorem} \ref{29}, we obtain some simple results for Parseval biframes.
\begin{cor}\label{46}
Suppose that $U,V\in B(\mathcal{H})$. For sequences $\{f_{k}\}_{k=1}^{\infty}$ and $\{g_{k}\}_{k=1}^{\infty}$ in $\mathcal{H}$, if each of the following conditions holds, then $(\{Uf_{k}\}_{k=1}^{\infty},\{Vg_{k}\}_{k=1}^{\infty})$ is a Parseval biframe if and only if  $VU^{*}=\textit{I}.$
\begin{enumerate}
\item[(i)] $\{f_{k}\}_{k=1}^{\infty}$ and $\{g_{k}\}_{k=1}^{\infty}$ are dual frames in $\mathcal{H}.$
\item[(ii)]  $\{f_{k}\}_{k=1}^{\infty}$ and $\{g_{k}\}_{k=1}^{\infty}$ are biorthogonal frames for $\mathcal{H}.$
\end{enumerate}
Moreover, if $\{e_{k}\}_{k=1}^{\infty}$ is an orthonormal basis for $\mathcal{H}$, then
$(\{Ue_{k}\}_{k=1}^{\infty},\{Ve_{k}\}_{k=1}^{\infty})$ is a Parseval biframe for $\mathcal{H}$ if and only if $ VU^{*}=\textit{I}.$
\end{cor}
\begin{proof}
Each of the conditions (i) and (ii) implies that   $(\{f_{k}\}_{k=1}^{\infty},\{g_{k}\}_{k=1}^{\infty})$ is a Parseval biframe. The proof of part (i) is clear. To see the part (ii), if $\{f_{k}\}_{k=1}^{\infty},\{g_{k}\}_{k=1}^{\infty}$ are biorthogonal frames, then for $f\in\mathcal{H}$, 
\begin{align*}
\sum_{k=1}^{\infty}\langle f,f_k\rangle\langle g_k,f\rangle &=\sum_{k=1}^{\infty}\langle f,f_k\rangle\langle g_k,\sum_{i=1}^{\infty}\langle f,S_{F}^{-1}f_i\rangle f_i\rangle\\
&=\sum_{k=1}^{\infty}\sum_{i=1}^{\infty}\langle f,f_k\rangle\langle S_F^{-1}f_i,f\rangle\langle g_k,f_i\rangle\\
&=\sum_{k=1}^{\infty}\sum_{i=1}^{\infty}\langle f,f_k\rangle\langle S_F^{-1}f_i,f\rangle{\delta}_{k,i}\\
&=\sum_{k=1}^{\infty}\langle f,f_k\rangle\langle S_F^{-1}f_k,f\rangle ={\Vert f\Vert}^{2}.
\end{align*}
Now, the desired result follows from the final assertion of \emph{Theorem} \ref{29}.\\
Also, the pair $(\{e_{k}\}_{k=1}^{\infty},\{e_{k}\}_{k=1}^{\infty})$ is a Parseval biframe and again, by the final assertion of \emph{Theorem} \ref{29}, the statement is true.
\end{proof}
\section{Charactrization of biframes}
\ In this section, we characterize the biframes whose constituent sequences are Bessel sequences, frames, and Riesz bases. Also, we obtain some results concerning Parseval biframes. 
  \begin{theorem}\label{33}
Two Bessel sequences $F=\{f_{k}\}_{k=1}^{\infty}$ and $G=\{g_{k}\}_{k=1}^{\infty}$ form a biframe for $\mathcal{H}$,  if and only if there exist operators $Q\in B^{+}_{b.b.}(\mathcal{H})$ and $T,W\in B(\mathcal{H})$ such that $TW^{*}=\textit{I}$, $f_{k}={Q}^{r}W e_{k}$, and $g_{k}={Q}^{t}T e_{k}$, for all $k\in\mathbb{N}$, whenever $r,t\in\mathbb{R}$ and $r+t=1$ and $E=\{e_k\}_{k=1}^{\infty}$ is an orthonormal basis for $\mathcal{H}$. 
\end{theorem}
\begin{proof}
Let $(F,G)=(\{f_{k}\}_{k=1}^{\infty},\{g_{k}\}_{k=1}^{\infty})$ be a biframe with biframe operator $S_{F,G}$, formed by Bessel sequences $\{f_{k}\}_{k=1}^{\infty}$ and $\{g_{k}\}_{k=1}^{\infty}$. Since $\{f_{k}\}_{k=1}^{\infty}$ and $ \{g_{k}\}_{k=1}^{\infty}$ are Bessel sequences, there exist operators $U,V\in B(\mathcal{H})$ and the orthonormal basis $E=\{e_k\}_{k=1}^{\infty}$ for $\mathcal{H}$, such that 
\begin{equation}\label{16}
F=f_{k}=Ue_{k}=UE\ \text{and}\ G=g_{k}=Ve_{k}=VE,\ \ \ \forall k\in\mathbb{N}.
\end{equation}
For the Parseval biframe $(\{e_{k}\}_{k=1}^{\infty},\{e_{k}\}_{k=1}^{\infty})$, \emph{Theorem} \ref{29} gives us operators 
 $T,W\in B(\mathcal{H})$ such that $TW^{*}=\textit{I}$, and by (\ref{14}),
$$ U=S^{r}_{UE,VE}W\ ,\ V=S^{t}_{UE,VE}T,$$ 
whenever $r,t\in\mathbb{R}$ and $r+t=1$. Now, by (\ref{16}), we can write
\begin{equation*}
f_{k}=Ue_{k}=S^{r}_{UE,VE}W e_{k}=S^{r}_{F,G}W e_{k},
\end{equation*}
and
\begin{equation}\label{17}
g_{k}=Ve_{k}=S^{t}_{UE,VE}T e_{k}=S^{t}_{F,G}T e_{k},\ \ \ \forall k\in\mathbb{N}.
\end{equation}
Conversely, suppose that  there exist an orthonormal basis $\{e_{k}\}_{k=1}^{\infty}$ and operators $Q\in B^{+}_{b.b.}(\mathcal{H})$ and $T,W\in B(\mathcal{H})$ such that $TW^{*}=\textit{I}$, $f_{k}={Q}^{r}W e_{k}$,  and $g_{k}={Q}^{t}T e_{k}$, for $k\in\mathbb{N}$, whenever $r,t\in\mathbb{R}$ and $r+t=1.$ For $f\in\mathcal{H}$, one obtains
\begin{align*}
\sum_{k=1}^{\infty}\langle f,f_{k}\rangle g_{k}&=\sum_{k=1}^{\infty}\langle f,{Q}^{r}We_{k}\rangle {Q}^{t}Te_{k}\\
&={Q}^{t}T\sum_{k=1}^{\infty}\langle {W}^{*}{Q}^{r}f,e_{k}\rangle e_{k}\\
&={Q}^{t}T{W}^{*}{Q}^{r}f \\ 
&={Q}^{t+r}f=Qf.
\end{align*}
These relations and \emph{Theorem} \ref{32}, show that $(\{f_{k}\}_{k=1}^{\infty},\{g_{k}\}_{k=1}^{\infty})$ is a biframe for $\mathcal{H}.$
\end{proof}
\begin{cor}
Two Bessel sequences $\{f_{k}\}_{k=1}^{\infty}$ and $\{g_{k}\}_{k=1}^{\infty}$ form a Parseval biframe for $\mathcal{H}$ if and only if 
there exist operators $T,W\in B(\mathcal{H})$ such that $T{W}^{*}=\textit{I}$, 
 $f_{k}=We_{k}$, and $g_{k}=Te_{k}$, for $k\in\mathbb{N}$, whenever $\{e_{k}\}_{k=1}^{\infty}$ is an orthonormal basis for $\mathcal{H}$.\\
 Furthermore, each of the above assertions implies that $\{f_{k}\}_{k=1}^{\infty}$ and $\{g_{k}\}_{k=1}^{\infty}$ are dual frames .
\end{cor}
Here we recall the concept of generalazied dual frames for $\mathcal{H}$ (g-dual frames). In 2013, Dehghan and Hasankhani \cite{Dehghan4} introduced the concept of g-dual frames as follows.\\
A frame  $\{g_{k}\}_{k=1}^{\infty}$ is a g-dual frame of frame $\{f_{k}\}_{k=1}^{\infty}$ for $\mathcal{H}$, if there exists an invertible operator $A\in B(\mathcal{H})$ such that for all $f\in\mathcal{H}$, the equality $f=\sum_{k=1}^{\infty}\langle Af,g_k\rangle f_k$ is valid.

The reconstruction formula (\ref{68}), for the members of a Hilbert space by a given biframe\\ $(\{f_{k}\}_{k=1}^{\infty},\{g_{k}\}_{k=1}^{\infty})$ for $\mathcal{H}$, and the concept of g-dual frames for $\mathcal{H}$ give us this idea that, characterize the biframes which are formed by two frames, by using the characterization of g-dual frames.
\begin{theorem}\label{69}
Let $F=\{f_{k}\}_{k=1}^{\infty}$ be a frame for $\mathcal{H}$ with frame operator $S_F$. Then the following statements are equivalent.
\begin{enumerate}
\item[(i)] The sequence $\{g_k\}_{k=1}^\infty$ is a Bessel for $\mathcal{H}$ and the pair $(\{f_k\}_{k=1}^\infty,\{g_k\}_{k=1}^\infty)$ is a biframe for $\mathcal{H}$.
\item[(ii)] $\{g_k\}_{k=1}^\infty=\{({S_FQ})^{-1}f_k+h_k-\sum_{j=1}^\infty\langle S^{-1}_F f_k,f_j\rangle h_j\}_{k=1}^\infty$, for some $Q\in B^{+}_{b.b.}(\mathcal{H})$, where $\{h_k\}_{k=1}^\infty$ is a Bessel sequence in $\mathcal{H}$.
\item[(iii)]  The sequence $\{g_k\}_{k=1}^\infty$ is a frame for $\mathcal{H}$ and the pair $(\{f_k\}_{k=1}^\infty,\{g_k\}_{k=1}^\infty)$ is a biframe for $\mathcal{H}$.
\end{enumerate}
\end{theorem}
\begin{proof}
For the proof of $(i)\Rightarrow (ii)$, suppose that $\{g_{k}\}_{k=1}^{\infty}$ is a frame for $\mathcal{H}$ and $(\{f_k\}_{k=1}^\infty,\{g_k\}_{k=1}^\infty)$ is a biframe for $\mathcal{H}$ with biframe operator $S_{F,G}$. By the reconstruction formula (\ref{68}) and according to this point that $S_{F,G}=S_{G,F}$, for $f\in\mathcal{H}$, we have
$$f=\sum_{k=1}^\infty\langle S^{-1}_{F,G}f,f_k\rangle g_k=\sum_{k=1}^\infty\langle S^{-1}_{F,G}f,g_k\rangle f_k.$$
Therefore, $\{f_k\}_{k=1}^\infty$ and $\{g_k\}_{k=1}^\infty$ are g-dual frames. By using the characterization of g-dual frames presented in \cite{Dehghan4}, every $g_k$ has the following form.
 \begin{align*}
 g_{k}&=(S_{F,G}S^{-1}_F)f_k+h_k-\sum_{j=1}^\infty\langle S^{-1}_F f_k,f_j\rangle h_j\\
 &={(S_FS^{-1}_{F,G})}^{-1}f_k+h_k-\sum_{j=1}^\infty\langle S^{-1}_F f_k,f_j\rangle h_j,
 \end{align*}
where the sequence $\{h_k\}_{k=1}^\infty$ is a Bessel sequence in $\mathcal{H}$.\\
For the proof of $(ii)\Rightarrow(i)$, assume that there exists an operator $Q\in B^{+}_{b.b.}(\mathcal{H})$ such that\\ $\{g_{k}\}_{k=1}^{\infty}=\{({S_FQ})^{-1}f_k+h_k-\sum_{j=1}^\infty\langle S^{-1}_F f_k,f_j\rangle h_j\}_{k=1}^\infty,$
where $\{h_k\}_{k=1}^\infty$ is a Bessel sequence in $\mathcal{H}$. For $f\in\mathcal{H}$, we have
\begin{align*}
\sum_{k=1}^\infty\langle f,f_k\rangle g_k &=\sum_{k=1}^\infty\langle f,f_k\rangle  \big({(S_FQ)}^{-1}f_k+h_k-\sum_{j=1}^\infty\langle S^{-1}_F f_k,f_j\rangle h_j\big)\\
&=Q^{-1}S^{-1}_F\sum_{k=1}^\infty\langle f,f_k\rangle f_k+\sum_{k=1}^\infty\langle f,f_k\rangle h_k-\sum_{k=1}^\infty\sum_{j=1}^\infty\langle f,f_k\rangle\langle S^{-1}_Ff_k,f_j\rangle h_j\\
&=Q^{-1}f+\sum_{k=1}^\infty\langle f,f_k\rangle h_k-\sum_{k=1}^\infty\sum_{j=1}^\infty\langle f,f_k\rangle\langle S^{-1}_Ff_k,f_j\rangle h_j\\
&=Q^{-1}f+\sum_{k=1}^\infty\langle f,f_k\rangle h_k-\sum_{j=1}^\infty\langle f,f_j\rangle h2_j=Q^{-1}f.
\end{align*}
Since $Q^{-1}\in B^{+}_{b.b.}(\mathcal{H})$, gives the result that $(\{f_k\}_{k=1}^\infty,\{g_k\}_{k=1}^\infty)$ is a biframe for $\mathcal{H}$. Simple calculations show that the sequence $\{g_k\}_{k=1}^\infty$ is a Bessel sequence. The above calculations derive that $f=\sum_{k=1}^\infty\langle Qf,f_k\rangle g_k$. Then, $\{g_k\}_{k=1}^\infty$ is a frame for $\mathcal{H}$ by (\cite{Dehghan4}, Lemma 2.1).\\
$(i)$ is an obvious consequence of $(iii)$.
\end{proof}
\begin{theorem}
Let $F=\{f_{k}\}_{k=1}^{\infty}$ be a Riesz basis for $\mathcal{H}$ with frame operator $S_F$. Then the following statements are equivalent.
\begin{enumerate}
\item[(i)] The pair $(\{f_k\}_{k=1}^\infty,\{g_k\}_{k=1}^\infty)$ is a biframe for $\mathcal{H}$.
\item[(ii)] $\{g_k\}_{k=1}^\infty=\{({S_FQ})^{-1}f_k\}$ for some $Q\in B^{+}_{b.b.}(\mathcal{H})$.\\
In particular, each of conditions $(i)$ and $(ii)$ implies that  $\{g_k\}_{k=1}^\infty$ is a Riesz basis for $\mathcal{H}$.
\end{enumerate}
\end{theorem}
\begin{proof}
For the proof of $(i)\Rightarrow(ii)$, let $(\{f_k\}_{k=1}^\infty,\{g_k\}_{k=1}^\infty)$ be a biframe for $\mathcal{H}$ with biframe operator $S_{F,G}$. Since $\{f_k\}_{k=1}^\infty$ is a Riesz basis for $\mathcal{H}$, there is an invertible operator $V\in B(\mathcal{H})$ such that $f_k=Ve_k$, for every $k\in\mathbb{N}$, where $\{e_k\}_{k=1}^\infty$ is an orthonormal basis for $\mathcal{H}$. By \emph{Theorem} \ref{39}, the sequence $\{g_k\}_{k=1}^\infty$ is a Riesz basis for $\mathcal{H}$ and for every $k\in\mathbb{N}$,
\begin{align*}
g_k&=S_{F,G}{{V}^{*}}^{-1}e_k=S_{F,G}{V^{*}}^{-1}{V}^{-1}f_k\\
&=S_{F,G}{(V{V}^{*})}^{-1}f_{k}=S_{F,G}{S_F}^{-1}f_k\\
&={(S_F{S_{F,G}}^{-1})}^{-1}f_k.
\end{align*}
So $\{g_k\}_{k=1}^\infty$ is presented form in (ii) because $S_{F,G}\in B^{+}_{b.b.}(\mathcal{H})$.\\
The proof of $(ii)\Rightarrow(i)$ is easily obtained. 
\end{proof}
\section{B-Riesz bases}
\ Since orthonormal bases are among the most important sequences in $\mathcal{H}$, we study those biframes for which one of the constituent sequences is an orthonormal basis. We will see that these types of biframes have interesting properties that distinguish them from the pair frames having a similar property. In fact, by collecting all the sequences that form a biframe together with a given orthonormal basis in a set, new bases can be obtained that find a special place between the set of all orthonormal bases and the set of all Riesz bases. But, such results cannot be established for pair frames. We begin our study of this subject with a definition.
\begin{defn}\label{47}
We consider an orthonormal basis $\{e_k\}_{k=1}^\infty$ for $\mathcal{H}$ and define the set $[\{e_k\}]$ as the following.
\begin{center}
$[\{e_k\}]=\Big\{\{f_k\}_{k=1}^\infty\vert (\{e_k\}_{k=1}^\infty,\{f_k\}_{k=1}^\infty) $ is a biframe for $\mathcal{H}\Big\}$.
\end{center}
\end{defn}
The next proposition shows us the way the elements of the set $[\{e_k\}]$ can be represented.
\begin{prop}\label{48}
Let $E=\{e_{k}\}_{k=1}^{\infty}$ be an orthonormal basis for $\mathcal{H}$. The sequence $F=\{f_{k}\}_{k=1}^{\infty}$ belongs to $[\{e_k\}]$ if and only if there exists an operator $U\in {B}^{+}_{b.b.}(\mathcal{H})$ such that  $f_{k}=Ue_{k}$, for all $k\in\mathbb{N}$. 
\end{prop}
\begin{proof}
Let $\{f_{k}\}_{k=1}^{\infty}\in[\{e_k\}]$. Therefore, the pair $(\{e_{k}\}_{k=1}^{\infty},\{f_{k}\}_{k=1}^{\infty})$ is a biframe with biframe operator $S_{E,F}$. For $k\in\mathbb{N}$,
$$S_{E,F}e_{k}=\sum_{i=1}^{\infty}\langle e_{k},e_{i}\rangle f_{i}=f_{k},$$
and hence $f_{k}=S_{E,F}e_{k}$, for all $k\in\mathbb{N}.$\\
Conversely, suppose that there exists $U\in {B}^{+}_{b.b.}(\mathcal{H})$ such that $f_{k}=Ue_{k}$, for all $k\in\mathbb{N}$. For $f\in\mathcal{H}$,
$$Uf=\sum_{k=1}^{\infty}\langle f,e_{k}\rangle Ue_{k}=\sum_{k=1}^{\infty}\langle f,e_{k}\rangle f_{k}=Sf.$$
 This means that $S\in{B}^{+}_{b.b.}(\mathcal{H})$, and by \emph{Theorem} \ref{32},  $(\{e_{k}\}_{k=1}^{\infty}, \{f_{k}\}_{k=1}^{\infty})$ is a biframe for $\mathcal{H}$ with biframe operator $S$. Hence, $\{f_{k}\}_{k=1}^{\infty}\in[\{e_k\}]$.
\end{proof}
\begin{prop}\label{49}
Let $\{e_{k}\}_{k=1}^{\infty}$ be an orthonormal basis for $\mathcal{H}$, and consider the sequence $\{f_{k}\}_{k=1}^{\infty}\in[\{e_k\}]$. If there exists an orthonormal basis $\{\delta_k\}_{k=1}^{\infty}$ for $\mathcal{H}$ such that $\{f_{k}\}_{k=1}^{\infty}\in[\{\delta_k\}]$, then $e_k=\delta_k$, for all $k\in\mathbb{N}$.
\end{prop}
\begin{proof}
Since the sequence $\{f_{k}\}_{k=1}^{\infty}$ belongs to $[\{e_k\}]$ and $[\{\delta_k\}]$, by \emph{Proposition} \ref{48}, there exist operators $U,V\in B^{+}_{b.b.}(\mathcal{H}) $ such that $f_{k}=Ue_k$ and $f_k=V\delta_k$, for all $k\in\mathbb{N}$.\\
On the other hand, the sequence $F=\{f_{k}\}_{k=1}^{\infty}$ is a frame for $\mathcal{H}$ with biframe operator $S_F$, too.
% \cite{chris}, (the operator $U$ is surjective). 
For $f\in\mathcal{H}$,
$$S_F f=\sum_{k=1}^{\infty}\langle f,f_k\rangle f_k=\sum_{k=1}^{\infty}\langle f,Ue_k\rangle Ue_k=U^2 f,$$
and also $S_F f=V^2 f.$
%$$S_F f=\sum_{k=1}^{\infty}\langle f,f_k\rangle f_k=\sum_{k=1}^{\infty}\langle f,V\delta_k\rangle V\delta_k=V^2 f.$$
These relations allow us to conclude that $U^2=V^2$, and so $U=V$. Now, for $k\in\mathbb{N}$, one obtains
$$f_k=Ue_k=U\delta_k.$$
Hence, $U(e_k-\delta_k)=0$, and so $e_k=\delta_k$.
\end{proof}
Here, a new class of sequences is introduced and we study elements of it.
\begin{defn}\label{70}
A sequence $\{f_{k}\}_{k=1}^{\infty}$ in $\mathcal{H}$ is called a \textbf{biframe-Riesz basis}, briefly \textbf{b-Riesz basis} for $\mathcal{H}$, if there is an orthonormal basis $\{e_{k}\}_{k=1}^{\infty}$ for $\mathcal{H}$ such that $(\{e_{k}\}_{k=1}^{\infty},\{f_{k}\}_{k=1}^{\infty})$ is a biframe for $\mathcal{H}$.
\end{defn}
Now, we can use \emph{Proposition} \ref{48} and \emph{Proposition} \ref{49}, and give some equivalent conditions for a sequence to be a b-Riesz basis.\\
\begin{theorem}\label{73}
Let $\{f_{k}\}_{k=1}^{\infty}$ be a sequence in $\mathcal{H}$. then the following statements are equivalent.
\begin{enumerate}
\item[(i)] $\{f_{k}\}_{k=1}^{\infty}$ is a b-Riesz basis.
\item[(ii)] $\{f_{k}\}_{k=1}^{\infty}\in [\{e_k\}]$ for an orthonormal basis $\{e_{k}\}_{k=1}^{\infty}$.
\item[(iii)] There exist an orthonormal basis $\{e_{k}\}_{k=1}^{\infty}$ and  an operator $U\in B^{+}_{b.b.}(\mathcal{H})$ such that $\{f_{k}\}_{k=1}^{\infty}=\{Ue_{k}\}_{k=1}^{\infty}$.
\item[(iv)] $\{f_{k}\}_{k=1}^{\infty}\in[\{e_k\}]$ for precisely an orthonormal basis $\{e_{k}\}_{k=1}^{\infty}.$
\end{enumerate}
\end{theorem}
%\begin{proof}
%$(i)\Leftrightarrow (ii)$. Clear by definition.\\
%$(ii)\Leftrightarrow (iii)$ is proved in \emph{Proposition} \ref{48}.\\
%$(iii)\Leftrightarrow (iv)$ is proved in \emph{Proposition} \ref{49}.
%\end{proof}
 If we denote the set of all orthonormal bases of $\mathcal{H}$ by $\mathcal{O}$, the set of all Riesz bases by $\mathcal{R}$, and the set of all b-Riesz bases by $\mathcal{E}$, then by \emph{Proposition} \ref{48}, these sets are ordered as $\mathcal{O}\subset\mathcal{E}\subset \mathcal{R}$. The following example illustrates that they are proper subsets.
\begin{example}\label{71}
\begin{enumerate}
\item[(i)] We consider the Riesz basis $\{f_k\}_{k=1}^{2}=\{(-1,2),(1,0)\}$ for ${\mathbb{R}}^2$.
 every orthonormal basis for ${\mathbb{R}}^2$ is in the following form. For $a\in[0,1]$,
 $$\{e^{a}_k\}_{k=1}^{2}=\Big\{(a,\sqrt{1-a^2}),(-\sqrt{1-a^2},a)\Big\}.$$
 If there exists $a\in[0,1]$ such that $\{f_k\}_{k=1}^{2}\in[\{e^{a}_k\}]$, then there are positive numbers $A$ and $B$ such that for every $(x,y)\in{\mathbb{R}}^2$
 \begin{align*}
A(x^2+y^2)&\leq\big\langle (x,y),(a,\sqrt{1-a^2})\big\rangle\big\langle (-1,2),(x,y)\big\rangle\\
&+\big\langle (x,y),(-\sqrt{1-a^2},a)\big\rangle\big\langle (1,0),(x,y)\big\rangle\\
 &\leq B(x^2+y^2),
\end{align*}
and then
 $$A(x^2+y^2)\leq -(a+\sqrt{1-a^2})x^2+2(\sqrt{1-a^2})y^2+(3a-\sqrt{1-a^2})xy\leq B(x^2+y^2).$$
 Now, set $y=1$. We get a quadratic equation as follows. 
 $$ -(a+\sqrt{1-a^2})x^2+(3a-\sqrt{1-a^2})x+2(\sqrt{1-a^2})=0,$$
 This equation has positive $\Delta=2a\sqrt{1-a^2}+9>0$. Hence, there are none zero points $(x,y)$ in ${\mathbb{R}}^2$ such that 
 $$ \big\langle(x,y),(a,\sqrt{1-a^2})\big\rangle\big\langle (-1,2),(x,y)\big\rangle +\\
 \big\langle (x,y),(-\sqrt{1-a^2},a)\big\rangle=0.$$
So there is no orthonormal basis for ${\mathbb{R}}^2$ such that form a biframe with $\{f_k\}_{k=1}^{2}$, that is $\{f_k\}_{k=1}^{2}$ is not belong to $\mathcal{E}$, hence $\mathcal{R}\neq\mathcal{E}$.
 \item[(ii)] We consider the orthonormal basis $\{e_k\}_{k=1}^2=\{(1,0),(0,1)\}$ for ${\mathbb{R}}^2$ and the sequence\\
  $\{f_k\}_{k=1}^2=\{(3,-1),(-1,2)\}$. The following calculations show that $\{f_k\}_{k=1}^2\in [\{e_k\}]$.\\
   For $(x,y)\in{\mathbb{R}}^2$,
  $$\big\langle(x,y),(1,0)\big\rangle\big\langle (3,-1),(x,y)\big\rangle +
 \big\langle (x,y),(0,1)\big\rangle\big\langle(-1,2),(x,y)\big\rangle=3x^2-2xy+2y^2.$$
 Now we can see that
  $$x^2+y^2\leq 3x^2-2xy+2y^2 \leq 4(x^2+y^2).$$
 So $(\{e_k\}_{k=1}^2,\{f_k\}_{k=1}^2)$ is a biframe for ${\mathbb{R}}^2$ with bounds 1 and 4. But it is clear that $\{f_k\}_{k=1}^{2}$ is not an orthonormal basis for ${\mathbb{R}}^2$, so $\mathcal{E}\neq\mathcal{O}$.
\end{enumerate}
\end{example}
The \emph{Proposition} \ref{49}, shows that the subsets $[\{e_k\}]$ of $\mathcal{E}$ are distinct whenever $\{e_k\}_{k=1}^\infty$ is an orthonormal basis for $\mathcal{H}$. Also, it is clear that $\mathcal{E}=\bigcup\big[\{e_{k}\}\big]$, hence the collection $\Gamma=\Big\{[\{e_k\}]\mid\{e_k\}$ is an orthonormal basis for $\mathcal{H}\Big\}$ is a partition of $\mathcal{E}$, and therefore induces an equivalence relation $\sim$ on the set $\mathcal{E}$, i.e., this result leads us to equivalence classes on $\mathcal{E}$. This is why we represented these subsets by the notation $[\{e_k\}]$. In fact, we have the following equivalence relation $\sim$ between sequences $\{f_{k}\}_{k=1}^{\infty}$ and $\{g_{k}\}_{k=1}^{\infty}$ in $\mathcal{E}$.\\
$\{f_{k}\}_{k=1}^{\infty}\sim\{g_{k}\}_{k=1}^{\infty}\Leftrightarrow$ there exists an orthonormal basis $\{e_{k}\}_{k=1}^{\infty}$ s.t. $\{f_{k}\}_{k=1}^{\infty},\{g_{k}\}_{k=1}^{\infty}\in[\{e_k\}]$.

In order to show another difference between biframes and pair frames, in what follows we study a similar problem, and we observe that the results obtained for biframes do not apply to pair frames.
\begin{defn}\label{52}
Consider the set $\mathcal{E}'$ of sequences in $\mathcal{H}$ as follows.
\begin{center}
$\mathcal{E}^{'}=\Big\{\{f_k\}_{k=1}^{\infty}\mid (\{e_k\}_{k=1}^{\infty},\{f_k\}_{k=1}^{\infty})$ is a pair frame, for some orthonormal basis $\{e_k\}_{k=1}^{\infty}\Big\}.$\\
\end{center}
For an orthonormal basis $\{e_k\}_{k=1}^{\infty}$ the subset  ${[\{e_k\}]}^{'}$ of $\mathcal{E}^{'}$ is considered as follows.
 \begin{center}
 ${[\{e_k\}]}^{'}=\Big\{\{f_k\}_{k=1}^{\infty}\mid (\{e_k\}_{k=1}^{\infty},\{f_k\}_{k=1}^{\infty})$ is a pair frame for $\mathcal{H}\Big\}$.
 \end{center}
 \end{defn}
 \begin{prop}\label{53}
 The set $\mathcal{E}^{'}$ is the set of all Riesz bases.
 \end{prop}
 \begin{proof}
 Let $\{f_k\}_{k=1}^{\infty}\in \mathcal{E}^{'}$.  Then, there exists an orthonormal basis $\{e_k\}_{k=1}^{\infty}$ such that  $ (\{e_k\}_{k=1}^{\infty},\{f_k\}_{k=1}^{\infty})$ is a pair frame for $\mathcal{H}$, that is, the pair frame operator $S$ is a well-defined and invertible operator. For $j\in \mathbb{N}$,
 $$Se_j=\sum_{k=1}^{\infty}\langle e_j,e_k\rangle f_k=f_j.$$
 Therefore, $f_k=Se_k$, for all $k\in\mathbb{N}$. So, $\{f_k\}_{k=1}^{\infty}$ is a Riesz basis for $\mathcal{H}$.\\
 Conversely, suppose that the sequence $\{f_k\}_{k=1}^{\infty}$ is a Riesz basis for $\mathcal{H}$. Then, there exists an invertible operator $V\in B(\mathcal{H})$ such that $f_k=Ve_k$, for all $k\in\mathbb{N}$. For $f\in\mathcal{H}$,
 $$Vf=\sum_{k=1}^{\infty}\langle f,e_k\rangle Ve_k=\sum_{k=1}^{\infty}\langle f,e_k\rangle f_k.$$
 This implies that $(\{e_k\}_{k=1}^{\infty},\{f_k\}_{k=1}^{\infty})$ is a pair frame for $\mathcal{H}$ with the pair frame operator $V$.
  \end{proof}
  The next example shows that unlike the subsets $[\{e_k\}]$ of $\mathcal{E}$, the subsets ${[\{e_k\}]}^{'}$ of $\mathcal{E}^{'}$ are not distinct.  Hence, the collection of them cannot be a partition for the set $\mathcal{E}^{'}=\mathcal{R}$.
  \begin{example}\label{54}
 Consider the following orthonormal basiss for ${\mathbb{R}}^2$.
  \begin{center}
   $\{e_k\}_{k=1}^{2}=\{(1,0),(0,1)\}$ and  $\{\delta_k\}_{k=1}^{2}=\{(0,1),(1,0)\}.$ 
   \end{center}  
   It is easy to check that the sequence $\{f_k\}_{k=1}^{2}=\{(0,1),(1,1)\}$ belongs to $[\{e_k\}]^{'}$ and $[\{\delta_k\}]^{'}$, because for every $k\in\mathbb{N}$,
  $$ f_k=\begin{bmatrix}
  0 & 1\\
  1 & 1
  \end{bmatrix}e_k
  \ ,\ f_k=\begin{bmatrix}
  1 & 0\\
  1 & 1
  \end{bmatrix}\delta_k.$$
  \end{example}
% In what follows, we examine the biframe property according to be or not to be a b-Ries basis of its constituent sequences.In this regard, the following questions arise.
In what follows, we examine the biframe property according to whether the constituent sequences are b-Riesz bases or are not. In this regard, the following questions arise.
 \begin{enumerate}
 \item[$(Q_4)$]  Is it necessarily true that any two b-Riesz bases form a biframe?
 \item[$(Q_5)$] Is there a biframe for which none of the constituent sequences are b- Riesz bases?
 \item[$(Q_6)$] Is there a biframe for which just one of constituent sequences is a b- Riesz basis?
\end{enumerate}
The next example shows that the answer to question $(Q_4)$ is negative.
 In \emph{Proposition} \ref{57}, we propose a necessary condition for such sequences to form a biframe. Also, we have a positive answer to question $(Q_5)$ and $(Q_6)$ that are illustrated  in \emph{Example} \ref{58} and \emph{Example} \ref{72}.
\begin{example}
We consider the orthonormal $\{e_k\}_{k=1}^2=\{(1,0),(0,1)\}$ for ${\mathbb{R}}^2$ and two sequences $\{f_k\}_{k=1}^2$ and $\{g_k\}_{k=1}^2$ as the following.
\begin{center}
$\{f_k\}_{k=1}^2=\{(3,1),(1,1)\}$,
$\{g_k\}_{k=1}^2=\{(2,-1),(-1,1)\}.$
\end{center}
Simple calculations show that $(\{e_k\}_{k=1}^2,\{f_k\}_{k=1}^2)$ is a biframe for ${\mathbb{R}}^2$ with bounds $\frac{1}{2}$ and $4$, and $(\{e_k\}_{k=1}^2,\{g_k\}_{k=1}^2)$ is a biframe for ${\mathbb{R}}^2$ with bounds $\frac{1}{4}$ and $3$. For $(x,y)\in{\mathbb{R}}^2$,
$$\big\langle (x,y),(3,1)\big\rangle\big\langle (2,-1),(x,y)\big\rangle +
 \big\langle (x,y),(1,1)\big\rangle\big\langle (-1,1),(x,y)\big\rangle =5x^2-xy.$$
Now set $(x,y)=(1,5)$, hence $5x^2-xy=0$, and this implies that above summation do not satisfy in upper and lower biframe bounds condition and so $(\{f_k\}_{k=1}^2,\{g_k\}_{k=1}^2)$ is not a biframe for ${\mathbb{R}}^2$.
\end{example}
 %\begin{example}\label{58}
%Let $\{e_k\}_{k=1}^{\infty}$ be an orthonormal basis for $\mathcal{H}$.
%The operators $U,V$ are considered as the following.\\
%\begin{center}
%\begin{matrix}
%U:\mathcal{H}\rightarrow\mathcal{H}\ ,&\  V:\mathcal{H}\rightarrow\mathcal{H}\\
% Ue_{k}=ke_{k}&Ve_{k}=\frac{1}{k}e_{k}
%\end{matrix}\\
%\end{center}
%Set $\{f_{k}\}_{k=1}^{\infty}=\{Ue_{k}\}_{k=1}^{\infty}=\{ke_{k}\}_{k=1}^{\infty}$ and $\{g_{k}\}_{k=1}^{\infty}=\{Ve_{k}\}_{k=1}^{\infty}=\{\frac{1}{k}e_{k}\}_{k=1}^{\infty}$.\\
%For $f\in\mathcal{H}$,\\
%$$S_{F,G}f=\sum_{k=1}^{\infty}\langle f,f_{k}\rangle g_{k}=\sum_{k=1}^{\infty}\langle f,ke_{k}\rangle \frac{1}{k}e_{k}=f$$
%Therefore $S_{F,G}=\textit{I}$. This means that $(\{f_{k}\}_{k=1}^{\infty}, \{g_{k}\}_{k=1}^{\infty})$ is a Parseval biframe. The operators $U,V$ commute by each other but the operator $U$ is not  bounded below and the operator $V$ is not bounded and so is not positive, hence none of the sequences  $\{f_{k}\}_{k=1}^{\infty}$ and $ \{g_{k}\}_{k=1}^{\infty}$ are in $[\{e_k\}]$.
%\end{example}
 \begin{prop}\label{57}
 Let $\{e_{k}\}_{k=1}^{\infty}$ be an orthonormal basis for $\mathcal{H}$, and
suppose that $\{f_{k}\}_{k=1}^{\infty}=\{Ue_{k}\}_{k=1}^{\infty}$ and $\{g_{k}\}_{k=1}^{\infty}=\{Ve_{k}\}_{k=1}^{\infty}$ are in $[\{e_k\}]$. If the operator $VU$ is positive, then $(\{f_{k}\}_{k=1}^{\infty}, \{g_{k}\}_{k=1}^{\infty})$ is a biframe for $\mathcal{H}$. 
\end{prop}
\begin{proof}
Since $\{f_{k}\}_{k=1}^{\infty}$ and $\{g_{k}\}_{k=1}^{\infty}$ are in $[\{e_k\}]$, by \emph{Theorem} \ref{73},  $U$ and $V$ are in ${B}^{+}_{b.b.}(\mathcal{H})$. For $f\in\mathcal{H}$,
$$VUf=VU(\sum_{k=1}^{\infty}\langle f,e_{k}\rangle e_{k})=\sum_{k=1}^{\infty}\langle f,Ue_{k}\rangle Ve_{k}=\sum_{k=1}^{\infty}\langle f,f_{k}\rangle g_{k}=Sf.$$
If the operator $VU$ is positive, then $S\in{B}^{+}_{b.b.}(\mathcal{H})$ and so  $(\{f_{k}\}_{k=1}^{\infty}, \{g_{k}\}_{k=1}^{\infty})$ is a biframe for $\mathcal{H}$. 
\end{proof}
%As we have seen in the \emph{example 4.4}, every pair of Riesz bases do not form a biframe neccesarily. The proof of this proposition illustrates that this subject satisfies by adding the positivity and commute conditions to represented operators of two Riesz bases.\\
%The next example illustrates that there are the sequences that form a biframe, and constructed by the act of an operator on an orthonormal basis $\{e_k\}_{k=1}^{\infty}$, and the oprators commute but, they are not in $[\{e_k\}]$. 
\begin{example}\label{58}
Let $\{e_k\}_{k=1}^{\infty}$ be an orthonormal basis for $\mathcal{H}$.
%\end{center}
We consider the sequences $\{f_{k}\}_{k=1}^{\infty}$ and $\{g_{k}\}_{k=1}^{\infty}$ defined by
\begin{center}
 $\{f_{k}\}_{k=1}^{\infty}=\{ke_{k}\}_{k=1}^{\infty}$,\\
  $\{g_{k}\}_{k=1}^{\infty}=\{\frac{1}{k}e_{k}\}_{k=1}^{\infty}$.
\end{center}
The pair $(\{f_{k}\}_{k=1}^{\infty}, \{g_{k}\}_{k=1}^{\infty})$ is a Parseval biframe. Since
%Set $\{f_{k}\}_{k=1}^{\infty}=\{Ue_{k}\}_{k=1}^{\infty}=\{ke_{k}\}_{k=1}^{\infty}$ and $\{g_{k}\}_{k=1}^{\infty}=\{Ve_{k}\}_{k=1}^{\infty}=\{\frac{1}{k}e_{k}\}_{k=1}^{\infty}$.\\
for $f\in\mathcal{H}$,
\begin{align*}
\sum_{k=1}^\infty\langle f, f_k\rangle\langle g_k,f\rangle &=\sum_{k=1}^\infty\langle f, ke_k\rangle\langle \frac{1}{k}e_k,f\rangle\\
&=\sum_{k=1}^\infty {\vert\langle f, e_k\rangle\vert}^2={\Vert f\Vert}^2.
\end{align*}
Of course, $\{f_{k}\}_{k=1}^{\infty}$ and $\{g_{k}\}_{k=1}^{\infty}$ are not in $[\{e_k\}]$, because the left summation in the equations 
$$\sum_{k=1}^\infty\langle f, e_k\rangle\langle f_k,f\rangle=\sum_{k=1}^\infty\langle f, e_k\rangle\langle ke_k,f\rangle=\sum_{k=1}^\infty k{\vert\langle f, e_k\rangle\vert}^2,$$
and
%has no upper bound, and also the first summation in the following equations has no lower bound.
$$\sum_{k=1}^\infty\langle f, e_k\rangle\langle g_k,f\rangle=\sum_{k=1}^\infty\langle f, e_k\rangle\langle \frac{1}{k}e_k,f\rangle=\sum_{k=1}^\infty \frac{1}{k}{\vert\langle f, e_k\rangle\vert}^2$$
have no upper bound and lower bound, respectively.

\end{example} 
\begin{example}\label{72}
We consider the orthonormal basis $\{e_k\}_{k=1}^\infty=\{(1,0),(0,1)\}$ for ${\mathbb{R}}^2$, and the sequences $\{f_k\}_{k=1}^\infty$ and $\{g_k\}_{k=1}^\infty$ as follows.
$$\{f_k\}_{k=1}^2=\{(3,-1),(-1,2)\},$$
$$\{g_k\}_{k=1}^2=\{(0,\frac{1}{5}),(-1,\frac{13}{5})\}.$$
We have seen in Example \ref{71}, part (ii) that $\{f_k\}_{k=1}^2\in [\{e_k\}]$. The following calculations show that $\{g_k\}_{k=1}^2\notin [\{e_k\}]$. For $(x,y)\in{\mathbb{R}}^2$,
 $$\big\langle(x,y),(1,0)\big\rangle\big\langle (0,\frac{1}{5}),(x,y)\big\rangle +
 \big\langle (x,y),(0,1)\big\rangle\big\langle(-1,\frac{13}{5}),(x,y)\big\rangle=-\frac{4}{5}xy+\frac{13}{5}y^2.$$
 Now, we set $(x,y)=(1,\frac{13}{4})$ which gives the result $-\frac{4}{5}xy+\frac{13}{5}y^2=0$, so $(\{e_k\}_{k=1}^2,\{f_k\}_{k=1}^2)$ is not a biframe for ${\mathbb{R}^2}$. Also, the following calculations show that $(\{f_k\}_{k=1}^2,\{g_k\}_{k=1}^2)$ is a biframe for ${\mathbb{R}^2}$ with bounds $\frac{1}{6}$ and 7.
 $$\big\langle(x,y),(3,-1)\big\rangle\big\langle (0,\frac{1}{5}),(x,y)\big\rangle +
 \big\langle (x,y),(-1,2)\big\rangle\big\langle(-1,\frac{13}{5}),(x,y)\big\rangle=x^2+5y^2-4xy,$$
and so
  $$\frac{1}{6}(x^2+y^2)\leq x^2+5y^2-4xy\leq 7(x^2+y^2).$$
\end{example}
In the next theorem, we characterize those biframes for which one of the constituent sequences is a b-Riesz basis for $\mathcal{H}$.
\begin{theorem}\label{59}
%Let $\{e_{k}\}_{k=1}^{\infty}$ be an orthonormal basis for $\mathcal{H}$. Suppose that the sequence $\{f_{k}\}_{k=1}^{\infty}$ is in $[\{e_k\}]$ and $\{g_{k}\}_{k=1}^{\infty}$ is a sequence in $\mathcal{H}$. Then, the following statements are equivalent.
Let $\{f_{k}\}_{k=1}^{\infty}$ be a b-Riesz basis for $\mathcal{H}$ and $\{g_{k}\}_{k=1}^{\infty}$ is a sequence in $\mathcal{H}$. Then, the following statements are equivalent.
\begin{enumerate}
\item[(i)] $(\{f_{k}\}_{k=1}^{\infty},\{g_{k}\}_{k=1}^{\infty})$ is a biframe for $\mathcal{H}$.
\item[(ii)]There exist $U,Q\in{B}^{+}_{b.b.}(\mathcal{H})$ such that $f_{k}=Ue_{k}$ and $g_{k}=Q{U}^{-1}e_{k}$, for all $k\in\mathbb{N}$, where $\{e_{k}\}_{k=1}^{\infty}$ is an orthonormal basis for $\mathcal{H}$.
\end{enumerate}
\end{theorem}
\begin{proof}
For the proof of $(i)\Rightarrow(ii)$, let $(\{f_{k}\}_{k=1}^{\infty},\{g_{k}\}_{k=1}^{\infty})$ be a biframe with biframe operator $S_{F,G}\in{B}^{+}_{b.b.}(\mathcal{H})$. $\{f_{k}\}_{k=1}$ is a b-Riesz basis and by \emph{Theorem} \ref{73}, there exist an orthonormal basis $\{e_{k}\}_{k=1}^{\infty}$ for $\mathcal{H}$, and an operator $U\in{B}^{+}_{b.b.}(\mathcal{H})$ such that $f_{k}=Ue_{k}$, for all $k\in\mathcal{H}$. For $f\in\mathcal{H}$, we obtain
$$S_{F,G}f=\sum_{i=1}^{\infty}\langle f,f_{i}\rangle g_{i}=\sum_{i=1}^{\infty}\langle f,Ue_{i}\rangle g_{i}=\sum_{i=1}^{\infty}\langle Uf,e_{i}\rangle g_{i},$$
and so for $k\in\mathbb{N}$,
\begin{equation*}\label{60}
S_{F,G}{U}^{-1}e_{k}=\sum_{i=1}^{\infty}\langle U{U}^{-1}e_{k},e_{i}\rangle g_{i}=\sum_{i=1}^{\infty}\langle e_{k},e_{i}\rangle g_{i}=g_{k}.
\end{equation*}
This implies the desired result.
Now, suppose that (ii) holds. For $f\in\mathcal{H}$,
$$\sum_{k=1}^{\infty}\langle f,f_{k}\rangle g_{k}=\sum_{k=1}^{\infty}\langle f,Ue_{k}\rangle Q{U}^{-1}e_{k}=
Q{U}^{-1}\sum_{k=1}^{\infty}\langle Uf,e_{k}\rangle e_{k}=Q{U}^{-1}Uf=Qf.$$
This equality implies that $S_{F,G}=Q\in{B}^{+}_{b.b.}(\mathcal{H})$, and $(\{f_{k}\}_{k=1}^{\infty},\{g_{k}\}_{k=1}^{\infty})$ is a biframe by \emph{Theorem} \ref{32}.
\end{proof}
\begin{cor}\label{61}
Let $\{f_{k}\}_{k=1}^{\infty}$ be a b-Riesz basis for $\mathcal{H}$ and $\{g_{k}\}_{k=1}^{\infty}$ be a sequence in $\mathcal{H}$. Then, the following statements are equivalent.
\begin{enumerate}
\item[(i)] $(\{f_{k}\}_{k=1}^{\infty},\{g_{k}\}_{k=1}^{\infty})$ is a Parseval biframe for $\mathcal{H}$.
\item[(ii)]There exists $U\in{B}^{+}_{b.b.}(\mathcal{H})$ such that $f_{k}=Ue_{k}$ and $g_{k}={U}^{-1}e_{k}$, for all $k\in\mathbb{N}$, where $\{e_{k}\}_{k=1}^{\infty}$ is an orthonormal basis for $\mathcal{H}$.
\end{enumerate}
\end{cor}
The end corollary illusterates that corresponding to every b-Riesz basis exists b-Riesz bases for $\mathcal{H}.$
\begin{cor}\label{62}
Let $\{f_{k}\}_{k=1}^{\infty}$ be a b-Riesz basis for $\mathcal{H}$. Then the canonical dual of $\{f_{k}\}_{k=1}^{\infty}$, and also, every biorthogonal frame to $\{f_{k}\}_{k=1}^{\infty}$, are b-Riesz bases for $\mathcal{H}$.
\end{cor}
\begin{proof}
Since $\{f_{k}\}_{k=1}^{\infty}$ is a Riesz basis, it only has one dual frame, the canonical dual frame. Now, according to the proof of Lemma \ref{46}, every dual of $\{f_{k}\}_{k=1}^{\infty}$, and also every biorthogonal frame to $\{f_{k}\}_{k=1}^{\infty}$, form a Parseval biframe with $\{f_{k}\}_{k=1}^{\infty}$. Hence, by \emph{Corollary} \ref{61}, they are b-Riesz bases for $\mathcal{H}$.
\end{proof}
%----------------------------------------------------------------------------------------------------------------------------------------

%%%%%%%%%%%%%%%%%%%%%%%%%%%%%%%%%%%%%%%
\newpage2
\begin{center}
\Large{Corrigendum to \\
 ''Biframes and some of their properties''}
\\
	{M. Firouzi Parizi, A. Alijani, M.A. Dehghan} \\ [0.2cm]
	\small {\textit{Department of Mathematics, Faculty of Mathematics, 
	\\
	Vali-e-Asr University of Rafsanjan, Rafsanjan, Iran}}\\
		
 \texttt{{\small E-mails:f75maryam@gmail.com; Alijani@vru.ac.ir; dehghan@vru.ac.ir }}
 \end{center}
%------------------------------------------------------------------------------------%

Upon further examination, it was found that there is an error in the paper [1]. 
The authors would like to publish a corrigendum separately from the published article. They believe that publishing such a corrigendum would be beneficial for readers of the article.\\

In the continuation of the studies conducted on the sets $\mathcal{E}$ and $\mathcal{R}$ introduced in section 6, it was found that $\mathcal{E}=\mathcal{R}$, and this result affects the results obtained after that. Therefore, we modify these results as follows.\\
1. In Example 6.6  part (i), the presented set 
$$\{e^{a}_k\}_{k=1}^{2}=\Big\{(a,\sqrt{1-a^2}),(-\sqrt{1-a^2},a)\Big\},\ a\in [0,1].$$
is only a part of orthonormal bases for $\mathbb{R}^2$. Because for $a=-\sqrt{\frac{2}{5}}$, the set $\{e^{-\sqrt{\frac{2}{5}}}_k\}_{k=1}^{2}$ is an orhtonormal basis for $\mathbb{R}^2$ and  $\{f_k\}_{k=1}^{2}\in [\{e^{-\sqrt{\frac{2}{5}}}_k\}]$.\\

 2. After this example, it was found that there is not a sequence $\{f_k\}_{k=1}^{\infty}\in \mathcal{R}$ such that $\{f_k\}_{k=1}^{\infty}\notin \mathcal{E}$. Hence, we change proposition 6.8, in the following form.\\
 \emph{Proposition 6.8}.  The sets $\mathcal{E}$ and $\mathcal{E}'$ are equal to the set of all Riesz bases.\\
 We add the proof of $\mathcal{R}\subset\mathcal{E}$ in the continuation of the proof of this proposition.\\
 Let $\{f_k\}_{k=1}^{\infty}\in \mathcal{R}$. So there exist an invertible operator $V\in B(\mathcal{H})$ and an orthonormal basis $\{e_k\}_{k=1}^{\infty}$ for $\mathcal{H}$ such that $f_k=Ve_k$, for every $ k\in\mathbb{N}$. By polar decomposition [2] of $V^*$, we have $V^*=W(VV^*)^\frac{1}{2}$, where $W$ is a unitary operator. Therefore $V=(VV^*)^{\frac{1}{2}}W^*=S^{\frac{1}{2}}W^*,$ where $S$ is the frame operator of $\{f_k\}_{k=1}^{\infty}$.\\
  Now set $\delta_k:=W^* e_k$, then $\{\delta_k\}_{k=1}^{\infty}$ is an orthonormal basis for $\mathcal{H}$ and we have
   $$f_k=S^{\frac{1}{2}}\delta_k,\ \ \forall k\in\mathbb{N}.$$
 This implies that $\{f_k\}_{k=1}^{\infty}\in [\{\delta_k\}]$, and so $\{f_k\}_{k=1}^{\infty}\in\mathcal{E}$.\\
 
 3. On page 20, after the proof of \emph{Proposition 6.8}, the authors have given descriptions about the subsets ${[\{e_k\}]}'$ of $\mathcal{E}'$. These results are changed by the above corrections.\\ 
 By \emph{Proposition 6.8}, $\mathcal{E}=\mathcal{R}=\mathcal{E'}$ and therefore the sets ${[\{e_k\}]}'$ are distinct too, so the collection of them is a partition for $\mathcal{R}$ in a complex Hilbert space. Hence there is no difference between pair frames and biframes in this property. \emph{Example 6.9}, illustrates that the sets ${[\{e_k\}]}'$ are not distinct in a real Hilbert space.
 
 \begin{center}
 {\bf REFRENCES}
 \end{center}
%\bibitem{m.f}
$[1]$ M. Firouzi Parizi, A. Alijani, and M.A. Dehghan, Biframes and some of their properties, Journal of Inequalities and Applications., 1(2022) 1-24.\\
% \bibitem{Kub}
$[2]$ C.S. Kubrusly, The Elements of Operator Theory, Birkh$\ddot{a}$user,  Brazil. 2010.


\begin{thebibliography}{99}
\bibitem{Aldroubi}
A. Aldroubi, C. Cabrelli, and U. Molter, Wavelets on irregular grids with arbitrary dilation matrices and frame atomics
for L2 ( Rd ), Appl. Comput. Harmon. Anal., 17 (2004) 119–140.
\bibitem{Dehghan1}
 A. Alijani, M.A. Dehghan, G-frames and their duals in Hilbert $C^*$-modules, B. Iran. Math. Soc., 38 (3) (2011) 567–580.
\bibitem{Dehghan2}
A. Alijani, M.A. Dehghan, *-Frames in Hilbert $C^*$-modules, U.P.B. Sci. Bull., Series A, 73 (4) (2011) 89–106.
\bibitem{Asgari}
M.S. Asgari, A. Khosravi, Frames and bases of subspaces in Hilbert spaces, J. Math. Anal. Appl., 308 (2005) 541–
553.
\bibitem{Dehghan3}
A. Askarizadeh, M.A. Dehghan, G-frames as special frames, Turk. J. Math., 37 (2013) 60–70.
\bibitem{Azarmi1}
H. Azarmi, R.A. kamyabi Gol, and M. Janfada, Pseudo frame multiresolution structure on locally
compact abelian groups, Wavelets and Linear Algebra, 3(2) (2016) 43–54.

\bibitem{Balazs.p}
P. Balazs, Basic definition and properties of Bessel multipliers, J. Math. Anal. Appl., 325 (2007) 571–585.
\bibitem{Balazs}
P. Balazs, Weighted and Controlled Frame, International Journal of Wavelets, Multiresolution and Information Processing,  8(1) (2010) 109–132. 

\bibitem{Bogda} 
I. Bogdanova, P. Vandergheynst, J.P. Antoine, L. Jacques, and M. Morvidone,
Stereographic wavelet frames on the sphere, Appl. Comput. Harmon. Anal., 19 (2005) 223–252.
\bibitem{Casa}
P. Casazza, G. Kutyniok, Frames of subspaces, Wavelets, frames and operator theory, Amer. Math. Soc., 345 (2004)  87–113.
%\bibitem{Casa2}
%P.G. Casazza, D. Han, and D.R. Larson, Frames for Banach spaces, Contemp. Math., 247 (1999) 149–182.
\bibitem{chris}
O. Christensen, An Introduction to Frames and Riesz Bases, Birkh$\ddot{a}$user,
Basel, Boston, Berlin, 2003.
%\bibitem{chris2}
%O. Christensen, M. Hasannasab, and E. Rashidi, Dynamical sampling and frame representations with bounded operators, J. Math. Anal. Appl., 463 (2018) 634–644.
\bibitem{Chris3}
O. Christensen, Y.C. Eldar, Oblique dual frames and shift-invariant spaces, Appl. Comput. Harmon. Anal., 17 (2004)
48–68.
\bibitem{D.G.M}
I. Daubechies, A. Grossmann, and Y. Meyer, Painless nonorthogonal expansions, J. Math. Phys., 27 (1986) 1271–1283.
\bibitem{Daub}
I. Daubechies, Ten Lectures on Wavelets, SIAM, Philadelphia, Pennsylvania, 1992.
\bibitem{Dehghan4}
M.A. Dehghan, M.A. Hasankhani Fard, G-dual frames in Hilbert spaces, U.P.B. Sci. Bull., Series A, 75 (2013) 129–140.
\bibitem{Duffin}
R.J. Duffin, A.C. Schaeffer, A class of nonharmonic Fourier series, Trans. Amer. Math. Soc., 72 (1952) 341–366.
%\bibitem{Heil}
%C. Heil. A Basis Theory Primer. Birkh$\ddot{a}$user. Boston. 2011.
\bibitem{Eldar}
Y. Eldar, Sampling with arbitrary sampling and reconstruction spaces and oblique dual frame vectors, J. Fourier
Anal. Appl., 9 (2003) 77–96.
\bibitem{Feich}
H.G. Feichtinger, T. Strohmer (Eds.), Gabor Analysis and Algorithms: Theory and applications, Birkh$\ddot{a}$user, Boston, MA, 1998.
\bibitem{Han}
D. Han, D.R. Larson, Frames, bases, and group representations, Mem. Amer. Math, Soc., 147 2(2000).2
\bibitem{safapour}
A. Fereydooni, A. Safapour, Pair frames, Results in Mathematics, 66 (2014) 247–263.
\bibitem{Forna1}
M. Fornasier, Quasi-orthogonal decompositions of structured frames, J. Math. Anal. Appl., 289 (2004) 180–199.
\bibitem{Forna2}
 M. Fornasier, Decompositions of Hilbert spaces: Local construction of global frames, (B. Bojanov, Ed.), DARBA, Sofia, (2003) 275-281.
\bibitem{Kub}
C.S. Kubrusly, The Elements of Operator Theory, Birkh$\ddot{a}$user,  Brazil. 2010.
\bibitem{Li}
S. Li, H. Ogawa, Pseudoframes for subspaces with applications, J. Fourier Anal. Appl., 10 (2004) 409–431.
\bibitem{Musa}
K. Musazadeh, H. Khandani, Some results on controlled frames in Hilbert spaces, Acta Mathematica Scientia, 36B(3) (2016) 655–665.

\bibitem{sun1}
W. Sun, G -frames and g-Riesz bases, J. Math. Anal. Appl., 322 (1) (2006) 437–452.
\bibitem{sun2}
W. Sun, Stability of g-frames, J. Math. Anal. Appl., 326 (2) (2006) 858–868.
\bibitem{Young}
R. Young, An Introduction to Non-Harmonic Fourier Series, Academic Press, New York, 1980.
\end{thebibliography}
\end{document}